\documentclass[11pt,british]{amsart}

\usepackage{babel,eucal,url,amssymb,enumerate,amscd,graphics}

\theoremstyle{plain}
\newtheorem{lemma}{Lemma}[section]
\newtheorem{definition}[lemma]{Definition}
\newtheorem{proposition}[lemma]{Proposition}
\newtheorem{corollary}[lemma]{Corollary}
\newtheorem{theorem}[lemma]{Theorem}
\newtheorem{remark}[lemma]{Remark}

\newtheorem*{ack}{Acknowledgements}

\newcommand{\Lie}[1]{\operatorname{\textsl{#1}}}

\newcommand{\Gtwo}{\ifmmode{{\rm G}_2}\else{${\rm G}_2$}\fi}

\def\sideremark#1{\ifvmode\leavevmode\fi\vadjust{\vbox to0pt{\vss
 \hbox to 0pt{\hskip\hsize\hskip1em
 \vbox{\hsize2.5cm\tiny\raggedright\pretolerance10000
 \noindent #1\hfill}\hss}\vbox to8pt{\vfil}\vss}}}%

\newfont{\eusm}{eusm10 scaled \magstep1}
\newfont{\eusmiii}{eusm10 scaled \magstep3}

\newcommand{\comp}{\makebox[7pt]{\raisebox{1.5pt}{\tiny $\circ$}}}
\newcommand{\RR}{\mbox{{\sl I}}\!\mbox{{\sl R}}}

\newcommand{\trace}{\mathop{\rm trace}}

\title[Existence of non-isotropic conjugate points]{Existence of non-isotropic conjugate points on rank one normal homogeneous spaces}

\author{J.~C.~Gonz{\'a}lez-D{\'a}vila}
\address[J.~C.~Gonz{\'a}lez-D{\'a}vila]{Department of Fundamental Mathematics\\
  University of La Laguna\\ 38200 La Laguna, Tenerife, Spain}
\email{jcgonza@ull.es}

\author{A.~M.~Naveira}
\address[A.~M.~Naveira]{Department of Geometry and Topology\\
  University of Valencia\\ 46100 Burjassot, Valencia, Spain}
\email{naveira@uv.es}

\date{\today}

\begin{document}
\maketitle

\begin{abstract}{\indent}
We give a positive answer to the Chavel's conjecture [J. Diff. Geom. 4 (1970), 13-20]: a simply connected rank one normal homogeneous space is symmetric if any pair of conjugate points are isotropic. It implies that all simply connected rank one normal homogeneous space with the property that the isotropy action is variational complete is a rank one symmetric space.
\vspace{4mm}

 \vspace{2mm}

\end{abstract}


\section{Introduction}\indent

The Jacobi equation for geodesics on a symmetric space has simple
solutions and one can directly show that every Jacobi field
vanishing at two points is the restriction of a Killing vector field
along the geodesic (see for example \cite{GS}). A Jacobi field $V$
on a Riemannian manifold $(M,g)$ which is the
restriction of a Killing vector field along a geodesic is called
{\em isotropic} \cite{Z}. For $(M,g)$ homogeneous Riemannian manifold, it means that $V$ is the restriction of
an infinitesimal motion of elements in the Lie algebra of the
isometry group $I(M,g)$ of $(M,g).$ Moreover, if $V$ vanishes at a
point $o$ of the geodesic then it is obtained as restriction of an
infinitesimal $\Lie{K}$-motion, being $\Lie{K}$ the isotropy subgroup of
$I(M,g)$ at $o\in M.$ This particular situation was what
originally motivated the term ``isotropic'' (see \cite{Ch} and
\cite{Ch1}).

Two points $p,q\in M$ are said to be {\em isotropically conjugate}
if there exists a non-zero isotropic Jacobi field $V$ along a
geodesic passing through $p$ and $q$ such that $V$ vanishes at
these points. When every Jacobi field vanishing at
$p$ and $q$ is isotropic, we say that they are {\em strictly
isotropic conjugate points.} Then {\em any pair of conjugate points in a Riemannian
symmetric space are strictly isotropic}.

 In the case of a
naturally reductive space, the adapted {\em canonical} connection has
the same geodesics and the Jacobi equation can be also written as
a differential equation with constant coefficients (equation
(\ref{Jam1})). Using this fact, Chavel in \cite{Ch} and \cite{Ch1} proved that the Berger spaces $B^{7}$ and $B^{13}$ admit conjugate points at which no isotropic Jacobi field vanishes. Such spaces are normal homogeneous of rank one, or equivalently, they have positive sectional curvature (see Lemma \ref{lrank}),  Moreover, in \cite{Ch2}, after studding conjugate points on odd-dimensional Berger spheres, he proposed the following conjecture:

{\em If every conjugate point of a simply connected normal homogeneous Riemannian manifold $\Lie{G}/\Lie{K}$ of rank one is isotropic, then $\Lie{G}/\Lie{K}$ is isometric to a Riemannian symmetric space of rank one.}

Here, our main purpose is to prove this conjecture. For it we develop a general theory about the existence of conjugate points which are not isotropic or not strictly isotropic along any geodesic on non-symmetric normal homogeneous spaces (see Theorem \ref{cp}).

The notion of {\em variationally complete action} was introduced by Bott and Samelson in \cite{Bott}. They proved that the {\em isotropy action} on a symmetric space of compact type is variationally complete. In \cite{GD} the first named author has proved that if the isotropy action of $K$ on $M = G/K$ is variationally complete then all Jacobi field vanishing at two points is $G$-isotropic. Then, the above conjecture implies the following:

 {\em If the isotropy action on a simply connected rank one normal homogeneous space is variationally complete then it is a compact rank one symmetric space.}

 Berger \cite{Be} has classified the homogeneous spaces $\Lie{G}/\Lie{K}$ which admit a normal $\Lie{G}$-invariant Riemannian metric with strictly positive sectional curvature. He states that a simply connected, normal homogeneous space of positive curvature is homeomorphic, in fact
diffeomorphic (see \cite[Proposition 4.3, Ch. II]{He}), to a
sphere $S^{n}$ or one of the projective spaces ${\mathbb C}P^{n},$
${\mathbb H}P^{n}$ or ${\mathbb C}aP^{2},$ with two exceptions: $B^{7}:= \Lie{Sp}(2)/\Lie{SU}(2)$  and $B^{13}:= \Lie{SU}(5)/\Lie{H},$ where $\Lie{H}$ is a Lie group isomorphic to $(\Lie{Sp}(2)\times
S^{1})/{\pm(id,1)}.$ The corresponding $\Lie{Sp}(2)$- and $\Lie{SU}(5)$-standard Riemannian metrics on $B^{7}$ and $B^{13}$ have positive sectional curvature. Moreover, the first one is, up to homotheties, the unique $Sp(2)$-invariant Riemannian metric on $B^{7},$ since it is isotropy irreducible.

 Wilking gives in \cite{Wil} a new quotient expression $W^{7}:= (\Lie{SO}(3)\times \Lie{SU}(3))/\Lie{U}^{\bullet}(2)$ for the positively curved seven-dimensional Aloff-Wallach spaces
$M^{7}_{11}$ \cite{AW}, equipped with a one-parameter family of
bi-invariant metrics on $\Lie{SO}(3)\times \Lie{SU}(3),$ turning them into normal homogeneous spaces.
Then, these spaces become a third exception in the Berger's
list.

It is worthwhile to note that the above classification is under
diffeomorphims and not under isometries. Then, in order to give a proof of the Conjecture, we shall need the following Riemannian classification for this class of spaces, based on results already known (see \cite{Wa}, \cite{Z} and \cite{Z1}) and where $\delta$ denotes the corresponding {\em pinching} constant.

\begin{theorem}\label{tmean2} A simply connected, normal homogeneous space of positive curvature is isometric to one of the following Riemannian spaces:
\begin{enumerate}
\item[{\rm (i)}] compact rank one symmetric spaces with their standard metrics: $S^{n}$ $(\delta =1);$ ${\mathbb
C}P^{n},$ ${\mathbb H}P^{n},$ ${\mathbb C}aP^{2}$ $(\delta =\frac{1}{4});$
\item[{\rm (ii)}] the complex projective space
${\mathbb C}P^{n} =
\Lie{Sp}(m+1)/(\Lie{Sp}(m)\times \Lie{U}(1)),$ $n = 2m+1,$ equipped with a standard
$\Lie{Sp}(m+1)$-homogeneous metric $(\delta = \frac{1}{16}).$
\item[{\rm (iii)}] the Berger spheres $(S^{2m+1} = \Lie{SU}(m+1)/\Lie{SU}(m),g_{s}),$ $0<s\leq 1,$ $(\delta(s) = \frac{s(m+1)}{8m - 3s(m+1)}).$
\item[{\rm (iv)}] $(S^{4m+3}= \Lie{Sp}(m+1)/\Lie{Sp}(m),g_{s}),$ $0<s\leq 1,$ $(\delta(s) = \frac{s}{8-3s},$ if $s\geq \frac{2}{3},$ and $\delta(s) = \frac{s^{2}}{4},$ if $s< \frac{2}{3}).$
\item[{\rm (v)}] $B^{7}= \Lie{Sp}(2)/\Lie{SU}(2)$ equipped with a standard $\Lie{Sp}(2)$-homogeneous metric.
\item[{\rm (vi)}] $B^{13}= \Lie{SU}(5)/\Lie{H}$ equipped with a standard $\Lie{SU}(5)$-homogeneous metric.
\item[{\rm (vii)}] $W^{7}= (\Lie{SO}(3)\times \Lie{SU}(3))/\Lie{U}^{\bullet}(2),$ equipped with a one-parameter family of $\Lie{SO}(3)\times \Lie{SU}(3)$-homogeneous metrics.
\end{enumerate}
\end{theorem}

 Eliasson \cite{Eli} and Heintze \cite{Hei} computed the pinching constants $\frac{1}{37}$ and $\frac{16}{29\cdot37}$ of $B^{7}$ and $B^{13}$ equipped with these normal homogeneous metrics. The pinching constant for any $SU(5)$-invariant metric on $B^{13}$ is obtained by P\"uttmann \cite{Put}. Moreover, he proves that the corresponding optimal pinching constant in $B^{13}$ and also in $W^{7}$ is $\frac{1}{37}.$

Up to the Berger space $B^{7},$ any non-symmetric rank one normal homogeneous space determines a homogeneous Riemannian fibration over a compact rank one symmetric space. This fact, together with the property that the isotropy action on the unit horizontal tangent sphere at the origin is transitive,  simplifies substantially the problem of determining conjugate points along any geodesic in theses spaces. It allows us to show in Theorems \ref{conj}, \ref{conjB13} and \ref{conjW7} the existence of conjugate points to the origin along any horizontal geodesic starting at this point which are not isotropic. Moreover, for normal homogeneous spaces of type (ii), (iii) and (iv) in Theorem \ref{tmean2} we prove that any geodesic admits conjugate points which are not strictly isotropic. Partial results are also given by geodesics in the Berger space $B^{13}$ and the Wilking example $W^{7}.$

\vspace{0.2cm}

\begin{ack} {\rm Supported by D.G.I. (Spain) and FEDER Project MTM2010-15444. Also the second author has been partially supported by the Generalitat Valenciana Project Prometeo 2009/099.}
\end{ack}

\section{Normal homogeneous spaces and isotropic Jacobi fields}\indent

Let $(M,g)$ be a connected homogeneous {\em Riemannian} manifold.
Then $(M,g)$ can be expressed as coset space $\Lie{G}/\Lie{K},$ where $\Lie{G}$ is a
Lie group, which is supposed to be connected, acting transitively
on $M,$ $\Lie{K}$ is the isotropy subgroup of $\Lie{G}$ at
some point $o\in M,$ the {\em origin} of $\Lie{G}/\Lie{K},$ and $g$ is a
$\Lie{G}$-invariant Riemannian metric. Moreover, we can assume that
$\Lie{G}/\Lie{K}$ is a {\em reductive homogeneous space}, i.e., there is an
${\mathrm A}{\mathrm d}(\Lie{K})$-invariant subspace ${\mathfrak m}$ of the Lie algebra ${\mathfrak
g}$ of $\Lie{G}$ such that ${\mathfrak g} = {\mathfrak k} \oplus {\mathfrak m},$
${\mathfrak k}$ being the Lie algebra of $\Lie{K}.$ $(M=\Lie{G}/\Lie{K},g)$ is said to
be {\em naturally reductive}, or more precisely $\Lie{G}$-{\em naturally
reductive}, if there exists a reductive decomposition ${\mathfrak g} = {\mathfrak k} \oplus
{\mathfrak m}$ satisfying
\begin{equation}\label{nred}
\langle[X,Y]_{\mathfrak m},Z\rangle + \langle [X,Z]_{\mathfrak m},Y\rangle = 0
\end{equation}
\noindent for all $X,Y,Z\in {\mathfrak m},$ where $[X,Y]_{\mathfrak m}$
denotes the ${\mathfrak m}$-component of $[X,Y]$ and $\langle\cdot,\cdot\rangle$ is the ${\mathrm A}{\mathrm d}({\Lie K})$-invariant inner product induced by $g$ on ${\mathfrak m},$ by using the canonical
identification ${\mathfrak m}\cong T_{o}M.$ When there exists an ${\mathrm A}{\mathrm d}(\Lie{G})$-invariant inner product on ${\mathfrak g},$ which we also denote by $\langle\cdot,\cdot\rangle,$ whose restriction to
${\mathfrak m} = {\mathfrak k}^{\bot}$ is $\langle\cdot,\cdot\rangle,$ the
space $(M=\Lie{G}/\Lie{K},g)$ is called {\em normal
homogeneous}. Then, for all $X,Y,Z\in {\mathfrak g},$ we have
\begin{equation}\label{tt2}
\langle [X,Y],Z\rangle + \langle [X,Z],Y\rangle = 0.
\end{equation}
Hence each normal homogeneous space is naturally reductive. A $\Lie{G}$-homogeneous Riemannian manifold is called {\em standard} if it is normal and $-\langle\cdot,\cdot\rangle$ is the {\em Killing}-{\em Cartan form} of ${\mathfrak g}.$ If $\Lie{G}$ is a simple compact Lie group, any naturally reductive $\Lie{G}$-homogeneous
Riemannian manifold is standard, up to scaling factor. Moreover, the unique $\Lie{G}$-invariant
 Riemannian metric, up to homotheties, on a compact isotropy irreducible space $M = \Lie{G}/\Lie{K}$
 is standard choosing the appropriate scaling factor and this Riemannian metric is Einstein.
Notice that not all standard homogeneous metrics are Einstein and not all normal homogeneous metrics are standard (see \cite[Ch. 7]{B} for more details).

Let $\tilde{T}$ denote the torsion tensor and $\tilde{R}$
the corresponding curvature tensor of the {\em canonical
connection} $\tilde{\nabla}$ of $(M,g)$ adapted to the reductive
decomposition ${\mathfrak g} = {\mathfrak m} \oplus {\mathfrak k}$
\cite[II, p. 190]{KN} defined by the sign convention $\tilde{R}(X,Y)
= \tilde{\nabla}_{[X,Y]} -
[\tilde{\nabla}_{X},\tilde{\nabla}_{Y}]$ and $\tilde{T}(X,Y) =
\tilde{\nabla}_{X}Y - \tilde{\nabla}_{Y}X - [X,Y],$ for all
$X,Y\in {\mathfrak X}(M),$ the Lie algebra of smooth vector fields
on $M.$ Then, these tensor fields at the origin are given by
\begin{equation}\label{torsion}
\tilde{T}_{o}(X,Y) = -[X,Y]_{\mathfrak m} \;\;\;, \;\;\;
\tilde{R}_{o}(X,Y) = {\rm ad}_{[X,Y]_{\mathfrak k}}
\end{equation}
and we have $\tilde{\nabla}g = \tilde{\nabla}\tilde{T} =
\tilde{\nabla}\tilde{R} = 0.$ On naturally reductive homogeneous
manifolds $(M=\Lie{G}/\Lie{K},g),$ the tensor field $S = \nabla -
\tilde{\nabla},$ where $\nabla$ denotes the Levi Civita connection
of $(M,g),$ is a {\em homogeneous structure} \cite{TV} satisfying
$S_{X}Y =-S_{Y}X = -\frac{1}{2}\tilde{T}(X,Y),$ for all $X, Y\in
{\mathfrak X}(M),$ and we get
\begin{equation}\label{RR}
\tilde{R}_{XY} = R_{XY} + [S_{X},S_{Y}] - 2S_{S_{X}Y}.
\end{equation}
From (\ref{torsion}) and
(\ref{RR}), the Riemannian curvature $R$ on a naturally reductive manifold satisfies $<R_{XY}X,Y> = <[[X,Y]_{\mathfrak k},X]_{\mathfrak m},Y> +
\frac{\textstyle 1}{\textstyle 4}\|[X,Y]_{\mathfrak m}\|^{2}$ and, moreover if it is normal, then
\begin{equation}\label{curnor}
<R_{XY}X,Y> = \|[X,Y]_{\mathfrak k}\|^{2} + \frac{\textstyle
1}{\textstyle 4}\|[X,Y]_{\mathfrak m}\|^{2},
\end{equation}
for all $X,Y\in {\mathfrak m}\cong T_{o}M.$ So, the sectional
curvature of a normal homogeneous manifold is always non-negative
and there exists a section $\pi=\RR\{X,Y\},$ $X,Y\in {\mathfrak
m},$ such that $K(\pi) = 0$ if and only if $[X,Y] = 0.$

The notion of Lie triple system given in the theory of symmetric spaces to construct totally geodesic submanifolds can be extended to naturally reductive spaces in the following way \cite{S}.
\begin{definition}{\rm Let $(M = \Lie{G}/\Lie{K},g)$ be a naturally reductive homogeneous
manifold with adapted reductive decomposition ${\mathfrak g} = {\mathfrak
m}\oplus {\mathfrak k}.$  A subspace $\nu$ of ${\mathfrak m}$
satisfying $[\nu,\nu]_{\mathfrak m} \subset {\nu}$ and $[[\nu,\nu]_{\mathfrak
k},\nu]\subset \nu$ is said to be a} Lie triple system {\rm (L.t.s.) of }${\mathfrak m}.$
\end{definition}

\begin{lemma} If $\nu\subset {\mathfrak m}$ is a L.t.s., then ${\mathfrak g}_{\nu} = \nu \oplus [\nu,\nu]_{\mathfrak k}$
is a Lie subalgebra of ${\mathfrak g}.$
\end{lemma}
\begin{proof} First, we need to check the following equality
\begin{equation}\label{corLts}
[[X,Y]_{\mathfrak k},[Z,W]_{\mathfrak k}] = [[X,[Z,W]_{\mathfrak
k}],Y]_{\mathfrak k} + [X,[Y,[Z,W]_{\mathfrak k}]_{\mathfrak k},
\end{equation}
for all $X,Y,Z,W\in {\mathfrak m}.$ Since ${\mathfrak m}$ is
$Ad(K)$-invariant, one obtains $[[X,Y]_{\mathfrak k},U] = [[X,Y],U]_{\mathfrak k},$ for all $X,Y\in {\mathfrak m}$ and $U\in {\mathfrak k}.$ Then,
using the Jacobi identity, $[[X,Y]_{\mathfrak k},U] = [[X,U],Y]_{\mathfrak k} +
[X,[Y,U]]_{\mathfrak k}.$ From here, putting $U = [Z,W]_{\mathfrak k},$ we obtain
(\ref{corLts}). Now, taking into account that $\nu$ is a L.t.s.,
$[\nu,\nu]_{\mathfrak k}$ is subalgebra of
${\mathfrak k}$ and the result is immediate.
\end{proof}
Denote $\Lie{G}_{\nu}$ the connected Lie subgroup of $\Lie{G}$ with Lie algebra
${\mathfrak g}_{\nu}.$ Then, in similar way than for symmetric spaces (see \cite[Theorem
7.2, Ch. IV]{He}), one obtains the following result, based on the fact of that geodesics of $M$
through $o$ are of type $(\exp tu)o,$ $u\in {\mathfrak m},$ and the corresponding totally geodesic submanifolds $M_{\nu}$
can be expressed as the orbit $\Lie{G}_{\nu}\cdot o$ (see \cite{S}).

\begin{proposition} Let $(M = \Lie{G}/\Lie{K},g)$ be a naturally reductive homogeneous
manifold with adap\-ted reductive decomposition ${\mathfrak g} =
{\mathfrak m} \oplus {\mathfrak k}$ and let $\nu\subset {\mathfrak m}$ be a
L.t.s. Then there exists a {\rm (}unique{\rm )} complete and
connected totally geodesic submanifold $M_{\nu}$ through $o$ such that $T_{o}M_{\nu} = \nu.$ Moreover, $M_{\nu}$ is the naturally
reductive homogeneous manifold $(M_{\nu}=\Lie{G}_{\nu}/(\Lie{G}_{\nu}\cap \Lie{K}), i^{*}g).$
\end{proposition}

\begin{remark}{\rm We are also interested on $M_{\nu}$ as a closed embedded
submanifold of $M.$ These conditions are satisfied when $\Lie{G}_{\nu}$
is a topological subgroup of $\Lie{G}$ and $\Lie{K}$ is compact (see \cite[Proposition 4.4, Ch. II]{He}).}
\end{remark}

From (\ref{curnor}), a flat totally geodesic submanifold of a normal homogeneous space $(M,g)$ is characterized by the property
that the corresponding Lie triple system is an abelian algebra. Defining {\em rank} of $(M,g)$ as the maximal dimension of a flat, totally geodesic submanifold, one directly
obtains
\begin{lemma}\label{lrank} The following statements are equivalent:
\begin{enumerate}
\item[{\rm (i)}] $(M,g)$ has positive sectional curvature;
\item[{\rm (ii)}] $[X,Y]\neq 0$ for all linearly independent $X,Y\in {\mathfrak m};$
\item[{\rm (iii)}] ${\rm rank}\;(M,g) = 1.$
\end{enumerate}
\end{lemma}

 For each $X\in {\mathfrak g},$ the mapping $\psi:\RR \times M \to
M,$ $(t,p)\in \RR\times M \mapsto \psi_{t}(p) = (\exp tX)p,$ is a
one-parameter group of isometries on $(M = \Lie{G}/\Lie{K},g)$ and consequently, $\psi$ induces
a Killing vector field $X^{*}$ given by $X^{*}_{p} = \frac{d}{dt}_{\mid t = 0}(\exp
tX)p.$ $X^*$ is called the {\em fundamental vector field} or the {\em
infinitesimal $\Lie{G}$-motion} corresponding to $X$ on $M.$ For any $a\in \Lie{G},$ we have
\begin{equation}\label{adjun1}
({\mathrm A}{\mathrm d}_{a}X)^{*}_{ap} = a_{*p}X^{*}_{p},
\end{equation}
where $a_{*p}$ denotes the differential map of $a$ at $p\in M.$

\begin{definition}{\rm A Jacobi field $V$ along a geodesic
$\gamma$ in $(M = \Lie{G}/\Lie{K},g)$ is said to be} $\Lie{G}$-isotropic {\rm if
there exists $X\in {\mathfrak g}$ such that $V = X^{*}\comp \gamma.$ }
\end{definition}
\noindent If $\Lie{G}$ is the identity connected component $I_{o}(M,g)$ of the isometry group $I(M,g)$ of $(M,g),$ then all (complete) Killing vector field on $M$ is a fundamental vector field $X^{*},$ for some $X\in
{\mathfrak g},$ and we simply say that $V$ is an {\em isotropic}
Jacobi field. Obviously, any $\Lie{G}$-isotropic Jacobi field is
isotropic but the converse does not satisfy in general \cite{GS}.

From the homogeneity of $M=\Lie{G}/\Lie{K},$ we shall only need to consider
geodesics $\gamma$ emanating from the origin $o\in M.$ In what follows, $\gamma_{u}$ will denote the unit-speed geodesic starting at $o$ with $\gamma_{u}'(0) = u\in {\mathfrak m},$ $\|u\|=1.$ Then a Jacobi field
$V$ along $\gamma_{u}$ with $V(0) = 0$ is $G$-isotropic if and only if
there exists $A\in {\mathfrak k}$ such that $V = A^{*}\comp \gamma_{u},$
or equivalently, if there exists an $A\in {\mathfrak k}$ such that \cite{GS}
\begin{equation}\label{ini}
(V(0),V'(0)) = (0,[A,u]).
\end{equation}

The linear isotropy representation, i.e. the differential of the action of $\Lie{K}$ on $T_{o}M,$ corresponds via the natural isomorphism $\pi_{*}:{\mathfrak m}\to T_{o}M$ with the Adjoint representation ${\mathrm A}{\mathrm d}(\Lie{K})$ of $\Lie{K}$ on ${\mathfrak m}.$ From (\ref{adjun1}) and since $k\comp \gamma_{u} =
\gamma_{k\cdot u},$ for any $k\in \Lie{K},$ where $k\cdot u = {\mathrm Ad}_{k}u,$  we can state:

\begin{lemma}\label{isotropic} If $V$ is an {\rm (}isotropic{\rm )} Jacobi field along $\gamma_{u}$, then $k_{*\gamma_{u}}V$ is an
{\rm (}isotropic{\rm )} Jacobi field along $\gamma_{k\cdot u}.$ Moreover, if $p = \gamma_{u}(t)$ is a {\rm ((}strictly{\rm )} isotropic{\rm
)} conjugate point to the origin along $\gamma_{u},$ then $k(p) =
\gamma_{k\cdot u}(t)$ is a {\rm ((}strictly{\rm )} isotropic{\rm )}
conjugate point to the origin along $\gamma_{k\cdot u}.$
\end{lemma}

On naturally reductive spaces, the connections $\nabla$ and $\tilde{\nabla}$ have the same geodesics and,
consequently, the same Jacobi fields (see \cite{Z}). Such
geodesics can be written as $\gamma_{u}(t) = (\exp tu)o$ and the Jacobi equation for $\nabla$
coincides with the Jacobi equation for $\tilde{\nabla},$
\[
\frac{\tilde{\nabla}^{2}V}{dt^{2}} -
\tilde{T}_{\gamma_{u}}\frac{\tilde{\nabla}V}{dt} + \tilde{R}_{\gamma_{u}}V
= 0,
\]
where $\tilde{R}_{\gamma_{u}} = \tilde{R}(\gamma'_{u},\cdot)\gamma'_{u}$ and
$\tilde{T}_{\gamma_{u}} = \tilde{T}(\gamma'_{u},\cdot).$ Taking into
account that $\tilde{\nabla}\tilde{T} = \tilde{\nabla}\tilde{R} =
0$ and the parallel translation with respect to $\tilde{\nabla}$
of tangent vectors at the origin along $\gamma_{u}$ coincides with
the differential of ${\exp}tu \in \Lie{G}$ acting on $M,$ it follows
that any Jacobi field $V$ along $\gamma_{u}(t)$ can be expressed as
$V(t) = (\exp tu)_{*o}X(t)$ where $X(t)$ is solution of the
differential equation
\begin{equation}\label{Jam1}
X''(t) - \tilde{T}_{u}X'(t) + \tilde{R}_{u}X(t) = 0
\end{equation}
in the vector space ${\mathfrak m},$ being $\tilde{T}_{u}X =
\tilde{T}(u,X) = -[u,X]_{\mathfrak m}$ and $\tilde{R}_{u}X =
\tilde{R}(u,X)u = [[u,X]_{\mathfrak k},u]$ (see \cite{GS},
\cite{Z} for more details).

The following characterization for isotropic Jacobi fields on normal
homogeneous spaces is so useful.
\begin{lemma}{\rm \cite{GD}}\label{inf} A Jacobi field $V$ along
$\gamma_{u}(t) = (\exp tu)o$ on a normal homogeneous space $(M
=\Lie{G}/\Lie{K},g)$ with $V(0) = 0$ is $\Lie{G}$-isotropic if and only if $V'(0)\in
({\rm Ker}\;\tilde{R}_{u})^{\bot}.$
\end{lemma}
We end this section establishing a more complete version of \cite[Proposition 2.3]{GD}. This is a key result for the development of this article.
\begin{theorem}\label{cp} Let $(M = \Lie{G}/\Lie{K},g)$ be a normal homogeneous space and let $u
,v$ be orthonormal vectors in ${\mathfrak m}$ verifying $[[u,v],u]_{\mathfrak m} = \lambda v,$ for some $\lambda >0.$
\begin{enumerate}
\item[{\rm (i)}] If $[u,v]_{\mathfrak m} = 0,$ then $\gamma_{u}(\frac{p\pi}{\sqrt{\lambda}}),$ for $p\in {\mathbb Z},$ are $\Lie{G}$-isotropically conjugate points to the origin.
\item[{\rm (ii)}] If $[u,v]\in {\mathfrak m}\setminus \{0\}$ and $[[u,[u,v]]_{\mathfrak k},u] =\rho[u,v],$ then $\rho = \frac{1}{\lambda}\|[[u,v],u]_{\mathfrak k}\|^{2}$ and the following cases hold:
\begin{enumerate}
\item[{\rm (A)}] If $\rho =0,$ i.e. $[[u,v],u]\in {\mathfrak m},$ $\gamma_{u}(\frac{2p\pi}{\sqrt{\lambda}}),$ $p\in {\mathbb Z},$ are conjugate points to the origin but not strictly $G$-isotropic.
\item[{\rm (B)}] If $\rho >0,$ $\gamma_{u}(\frac{s}{\sqrt{\lambda + \rho}}),$ where
\begin{enumerate}
\item[{\rm 1.}] $s$ is a solution of the equation $\tan
\frac{s}{2} = -\frac{\rho s}{2\lambda},$ or
\item[{\rm 2.}] $s = 2p\pi,$ $p\in {\mathbb Z},$
\end{enumerate}
are conjugate points to the origin along $\gamma_{u}(t) = (\exp tu)o.$ In the first case, they are not strictly $\Lie{G}$-isotropic and in the second one, they are $G$-isotropic.
\end{enumerate}
\end{enumerate}
\end{theorem}

\begin{proof} (i) From (\ref{torsion}) we get
$\tilde{T}_{u}v = 0$ and $\tilde{R}_{u}v = \lambda v.$ Then, $X(t) = A\sin\sqrt{\lambda}tv$ is a solution of
(\ref{Jam1}) with $X(0) = 0.$ Because from (\ref{tt2}) $\tilde{R}_{u}$ is
self-adjoint, $v\in ({\rm Ker}\;\tilde{R}_{u})^{\bot}$ and using
Lemma \ref{inf}, $V(t) = (\exp tu)_{*o}X(t)$ is $\Lie{G}$-isotropic.

\noindent (ii) Here, we obtain
$$
\begin{array}{lcllcllcl}
\tilde{T}_{u}v  & = &  -\sqrt{\lambda} w,& &\tilde{T}_{u}w &  = & \sqrt{\lambda}v,\\[0.3pc]
\tilde{R}_{u}v  & = &  0, & & \tilde{R}_{u}w & = & \rho w,
\end{array}
$$
where $w = \frac{1}{\sqrt{\lambda}}[u,v].$ From
(\ref{nred}) and (\ref{tt2}) it follows that $\lambda = \|[u,v]\|^{2}$ and $\rho = \|[u,w]_{\mathfrak k}\|^{2},$ which implies that $\rho =<\tilde{R}_{u}w,w>.$ Hence, the solutions $X(t) =
X^{1}(t) v + X^{2}(t)w$ of (\ref{Jam1}) satisfy
$$
\left \{
\begin{array}{l}
{X^{1}}'' - \sqrt{\lambda}{X^{2}}' = 0,\\
{X^{2}}'' + \sqrt{\lambda} {X^{1}}' + \rho X^{2}= 0.\\
\end{array}
\right.
$$
Now, differentiating the second equation and substituting
${X^{1}}''$ from the first one, we get
\[
{X^{2}}''' + (\lambda + \rho){X^{2}}' = 0.
\]
Therefore, $X(t)$ with $X(0) = 0$ is given by
\begin{eqnarray}\label{f1}
X(t) & = & \sqrt{\frac{\lambda}{\lambda + \rho}}\Big (A( 1 -\cos
\sqrt{\lambda + \rho}t) - B(\frac{\rho\sqrt{\lambda +
\rho}}{\lambda} t + \sin \sqrt{\lambda + \rho}t) \Big )v\\
\nonumber & & + \Big (A\sin\sqrt{\lambda + \rho}t + B(1 - \cos
\sqrt{\lambda +\rho}t)\Big )w,
\end{eqnarray}
where $A,B$ are constant. Then $X'(0) = -B\frac{\lambda + \rho}{\sqrt{\lambda}}v + A\sqrt{\lambda
+ \rho}w$ and from Lemma \ref{inf}, the Jacobi vector fields $V(t) = (\exp tu)_{*o})X(t)$ along $\gamma_{u}$ with $B\neq 0$ are not $G$-isotropic. Hence, if $\rho =0,$ one gets $X(\frac{2p\pi}{\sqrt{\lambda}}) =0,$ for $p\in {\mathbb Z},$ and it proves (A).

Next, suppose $\rho>0.$ The values of $t$ such that $X(t) = 0,$ $t\neq 0,$
for some $A,B\in \RR,$ are the zeros of the determinant
$$
\left |
\begin{array}{lcl}
1- \cos\sqrt{\lambda + \rho}t & & - (\frac{\rho\sqrt{\lambda +
\rho}}{\lambda} t + \sin{\sqrt{\lambda +\rho}t})\\
\sin\sqrt{\lambda + \rho}t & & 1 - \cos\sqrt{\lambda + \rho}t
\end{array}
\right | ,
$$
that is, the zeros of the function $f(s) = 1 - \cos s -\mu s\sin s,$ where $s = s(t) = \sqrt{\lambda +
\rho}t$ and $\mu = -\frac{\rho}{2\lambda}.$
Hence, it follows that either $\sin s=0,$ which gives $s =
2p\pi,$ $p\in {\mathbb Z},$ or,
\[
\cos s =\frac{1-\mu^{2}s^{2}}{1 + \mu^{2}s^{2}},\;\;\;\;\; \sin s
= \frac{2\mu s}{1 + \mu^{2}s^{2}},
\]
which yields to the equation $\tan \frac{s}{2} = \mu s.$ Let $c$ be a solution of $\tan \frac{s}{2} = \mu s.$ Then $f(c) =0$ and it implies, substituting in (\ref{f1}), that $A = -\mu cB.$ So the vector fields $V(t)$ along $\gamma_{u}$
given by
$$
\begin{array}{lcl}
V(t)&  = & (\exp tu)_{*o}\Big ( \sqrt{\frac{\lambda}{\lambda +
\rho}} \Big (\mu c(1 - \cos s(t)) + (2 \mu s(t) + \sin s(t))\Big )v\\[0.6pc]
& &\hspace{0.3cm} +  (\mu c\sin s(t) -( 1 - \cos s(t)))w\Big ),
\end{array}
$$
are Jacobi fields such that $V(\frac{c}{\sqrt{\lambda +\rho}})
= 0$ and they are not $G$-isotropic.

Finally, because $\tilde{R}_{u}w = \rho w$ one gets that $w\in ({\mathrm Ker}\;\tilde{R}_{u})^{\bot}$ and hence, from Lemma \ref{inf}, the vector fields along
$\gamma_{u}$ spanned by
\[
V(t)  = (\exp tu)_{*o}\Big (\sqrt{\frac{\lambda}{\lambda + \rho}}
(1 -\cos s)v + \sin s w \Big )
\]
are $G$-isotropic Jacobi fields with $V(\frac{2p\pi}{\sqrt{\lambda
+ \rho}}) = 0,$ for all $p\in {\mathbb Z}.$
\end{proof}

\begin{remark} {\rm The equation $\tan
\frac{s}{2} = - \frac{\rho s}{2\lambda}$ has for $\rho >0$ a
solution $s_{o}\in ]\pi,2\pi[.$ Therefore,
$\gamma_{u}(\frac{s_{o}}{\sqrt{\lambda + \rho}})$ is the first
conjugate point of $\gamma_{u}$ to the origin among those here
obtained and then the conjugate radius is $\leq
\frac{s_{o}}{\sqrt{\lambda + \rho}}.$ }
\end{remark}
\section{Normal homogeneous metrics of positive curvature on symmetric spaces}\indent
A homogeneous Riemannian manifold with positive sectional curvature is compact and if it is moreover simply connected, it can be written as $\Lie{G}/\Lie{K}$ with $\Lie{G}$ a compact Lie group. Wallach \cite[Theorem 6.1]{Wa} showed  that
for the even-dimensional case, $\dim M =2n,$ it is
isometric to
\begin{enumerate}
\item[{\rm (i)}] a compact rank one symmetric space: ${\mathbb
C}P^{n},$ $S^{2n},$ ${\mathbb H}P^{n/2}$ $(n$ even{\rm )},
${\mathbb C}aP^{2}$ $(n =8);$

\noindent or

\item[{\rm (ii)}] one of the following quotient spaces $G/K$ with
a suitable $G$-invariant metric:
\begin{enumerate}
\item[{\rm (1)}] the manifolds of flags ${\mathbb F}^{6},$ ${\mathbb F}^{12}$ and ${\mathbb F}^{24}$ in ${\mathbb C}P^{2},$ ${\mathbb H}P^{2}$ and ${\mathbb C}aP^{2},$ respectively:
$$
\begin{array}{lcl}
{\mathbb F}^{6} & = & \Lie{SU}(3)/(S(\Lie{U}(1)\times \Lie{U}(1)\times \Lie{U}(1)),\\
{\mathbb F}^{12} & = & \Lie{Sp}(3)/(\Lie{SU}(2)\times \Lie{SU}(2)\times \Lie{SU}(2)),\\
{\mathbb F}^{24} & = & F_{4}/\Lie{Spin}(8);
\end{array}
$$
\item[{\rm (2)}] ${\mathbb C}P^{n} = \Lie{Sp}(m+1)/(\Lie{Sp}(m)\times \Lie{U}(1)),$
$n = 2m+1;$ \item[{\rm (3)}] the six-dimensional sphere $S^{6} =
\Lie{G}_{2}/\Lie{SU}(3).$
\end{enumerate}
\end{enumerate}
\begin{remark}{\rm In (i), the quotient spaces determined by the pairs $(\Lie{SU}(n+1),\Lie{S}(\Lie{U}(n)\times U(1))),$ $(\Lie{Spin}(2n+1),\Lie{Spin}(2n)),$ $(\Lie{Sp}(n),\Lie{Sp}(n-1)\times\Lie{Sp}(1))$ and $(\Lie{F}_{4},\Lie{Spin}(9))$ are isotropy-irreducible and so, they admit, up to homotheties, a unique invariant Riemannian metric.}
\end{remark}
Valiev \cite{Va} determined the set of all homogeneous Riemannian metrics on ${\mathbb
F}^{6},$ ${\mathbb F}^{12}$ and ${\mathbb F}^{24}$ of positive sectional curvature and their corresponding optimal
pinching constants. (Pinching constants means how much the local geometry of a compact Riemannian manifold $(M,g)$ with positive sectional curvature $K$ deviates from the geometry of a standard sphere. They are defined as quotients $\delta(M,g) = \frac{{\rm min}\;K}{{\rm max}\; K}$ of the extremal values of the sectional curvature.) According to the Berger's classification, they
cannot be normal homogeneous. It is worthwhile to note that $S^{6}
= \Lie{G}_{2}/\Lie{SU}(3)$ carries the usual metric of constant
sectional curvature \cite{B1}, the isotropy action of $\Lie{SU}(3)$ is irreducible on the
tangent space (see \cite{B}) and it is a nearly-K\"ahler $3$-symmetric space but $(\Lie{G}_{2},\Lie{SU}(3))$ is not a symmetric pair.

The complex projective space ${\mathbb C}P^{n},$ $n = 2m+1,$ equipped with the standard
$\Lie{Sp}(m+1)$-homogeneous
 Riemannian manifold, is also an irreducible compact nearly K\"ahler $3$-symmetric space \cite{G}. It can be viewed as the base space of the Hopf fibration
 \[
 S^{1}\to S^{4m+3} = \Lie{Sp}(m+1)/\Lie{Sp}(m) \to {\mathbb C}P^{n} = \Lie{Sp}(m+1)/(\Lie{Sp}(m)\times \Lie{U}(1)).
 \]
Here, the inclusion of $\Lie{Sp}(m)$ in $\Lie{Sp}(m+1)$ is the standard one. Taking $\Lie{Sp}(m+1)$ as Lie subgroup of $\Lie{SU}(2(m+1))$ also in the natural way, we obtain a reductive decomposition ${\mathfrak s}{\mathfrak p}(m+1) = {\mathfrak s}{\mathfrak p}(m) \oplus {\mathfrak m}$ adapted to the quotient $\Lie{Sp}(m+1)/\Lie{Sp}(m),$ with ${\mathfrak m} = {\mathfrak m}_{0} \oplus {\mathfrak m}_{1},$ being ${\mathfrak m}_{0}= {\mathfrak s}{\mathfrak u}(2) \cong {\mathfrak s}{\mathfrak p}(1)$ and ${\mathfrak m}_{1}\cong {\mathbb H}^{m}.$ The isotropy Lie subgroup $\Lie{Sp}(m)$ acts trivially on ${\mathfrak m}_{0}$ and by its standard representation on ${\mathfrak m}_{1}.$

Let $E_{ij}$ denote the square matrix on ${\mathfrak s}{\mathfrak u}(2(m+1))$ with entry $1$ where the
$i$th row and the $j$th column meet, all other entries being $0,$ and set
\begin{equation}\label{ABC}
\begin{array}{lcl}
A_{jk} &  = & \sqrt{-1}(E_{jj} -E_{kk}),\\[0.4pc]
B_{jk} & = &  E_{jk} - E_{kj},\\[0.4pc]
C_{jk} & = & \sqrt{-1}(E_{jk} +E_{kj}).
\end{array}
\end{equation}
In what follows we shall need the Lie  multiplication table:
\begin{equation}\label{bracABC}
\begin{array}{lcl}
[A_{rj},A_{kl}] & = & 0,\\[0.4pc]
[A_{rj},B_{kl}] & = & \delta_{rk}C_{rl} - \delta_{rl}C_{rk} - \delta_{jk}C_{jl} + \delta_{jl} C_{jk},\\[0.4pc]
[A_{rj},C_{kl}] & = & -\delta_{rk}B_{rl} - \delta_{rl}B_{rk} + \delta_{jk}B_{jl} + \delta_{jl}B_{jk},\\[0.4pc]
[B_{rj},B_{kl}] & = & \delta_{jk}B_{rl} - \delta_{jl}B_{rk} - \delta_{rk}B_{jl} + \delta_{rl}B_{jk},\\[0.4pc]
[B_{rj},C_{kl}] & = & \delta_{jl}C_{rk} + \delta_{jk}C_{rl} - \delta_{rl}C_{jk} - \delta_{rk}C_{jl},\\[0.4pc]
[C_{rj},C_{kl}] & = & -\delta_{jk}B_{rl} - \delta_{jl}B_{rk} - \delta_{rk}B_{jl} - \delta_{rl}B_{jk}.
\end{array}
\end{equation}
The canonical basis $Y_{1},\dots ,Y_{m}$ of ${\mathbb H}^{m}$ over ${\mathbb H}$ is given by $Y_{\alpha} = B_{\alpha,2m+1} + B_{m+\alpha,2(m+1)},$ for $\alpha = 1,\dots, m.$ If we put $X_{1} = i = A_{2m+1,2(m+1)},$ $X_{2} = j= B_{2m+1,2(m+1)}$ and $X_{3} = k = C_{2m+1,2(m+1)}$ as basis of ${\mathfrak m}_{0}$ then the corresponding basis $\{Y_{\alpha}; Y_{\alpha 1};Y_{\alpha 2}; Y_{\alpha 3}\}_{\alpha=1}^{m}$ over $\RR,$ where $Y_{\alpha 1} = iY_{\alpha},$ $Y_{\alpha 2} = jY_{\alpha},$ $Y_{\alpha 3} = k Y_{\alpha}$ is obtained by using of Lie brackets on ${\mathfrak m}.$ For $p,q,r$ a cyclic permutation of $1,2,3,$ one gets:
\begin{equation}\label{bracS}
 \begin{array}{lcl llcl l lcl}
 [X_{p},X_{q}] & = & 2 X_{r}, &  [X_{p},Y_{\alpha}] & = & - Y_{\alpha p},  & [X_{p},Y_{\alpha p}] & = &  Y_{\alpha},\\[0.4pc]
[X_{p}, Y_{\alpha q}] & = & Y_{\alpha r},  &  [Y_{\alpha},Y_{\beta}] & = & -Z_{\alpha,\beta}, &  [Y_{\alpha}, Y_{\alpha p}] & = & -2X_{p} + 2 Z_{\alpha p},\\[0.4pc]
[Y_{\alpha p},Y_{\alpha q}] & = & 2X_{r} + 2Z_{\alpha r}, & & & & & &
\end{array}
 \end{equation}
where $Z_{\alpha,\beta} = B_{\alpha,\beta} + B_{m + \alpha, m+ \beta},$ $1\leq \alpha < \beta\leq m,$ and $Z_{\alpha 1} = A_{\alpha,m+\alpha},$ $Z_{\alpha 2} = B_{\alpha,m+\alpha}$ and $Z_{\alpha 3} = C_{\alpha, m+\alpha},$ $1\leq \alpha \leq m.$ Moreover, for $\alpha\neq \beta,$
\begin{equation}\label{bracS1}
[Y_{\alpha},Y_{\beta p}] =Z_{(\alpha,\beta)p},\;\;\;\;\; [Y_{\alpha p},Y_{\beta p}] = -Z_{\alpha,\beta},\;\;\;\;\; [Y_{\alpha p},Y_{\beta q}] = Z_{(\alpha,\beta)r},
\end{equation}
where $Z_{(\alpha,\beta)1} = C_{\alpha,\beta} - C_{m+\alpha,m+\beta},$ $Z_{(\alpha,\beta)2} = B_{\alpha,m+\beta} + B_{\beta,m+\alpha}$ and $Z_{(\alpha,\beta)3}= C_{\alpha,m+\beta} + C_{m+\alpha,\beta}.$

A reductive decomposition of ${\mathfrak s}{\mathfrak p}(m+1)$ adapted to $\Lie{Sp}(m+1)/(\Lie{Sp}(m)\times \Lie{U}(1))$ is then ${\mathfrak s}{\mathfrak p}(m+1) = ({\mathfrak s}{\mathfrak p}(m)\oplus {\mathfrak u}(1))\oplus {\mathfrak p},$ where ${\mathfrak u}(1)$ is generated by $X_{1}$ and ${\mathfrak p} = {\mathfrak p}_{0} \oplus {\mathfrak p}_{1},$ being ${\mathfrak p}_{0}$ the subspace generated by $X_{2}$ and $X_{3}$ and ${\mathfrak p}_{1} = {\mathfrak m}_{1}\cong {\mathbb H}^{m}.$ In fact, ${\mathfrak p}_{0}$ is a Lie triple system with corresponding totally geodesic submanifold through the origin the $2$-sphere $\Lie{Sp}(1)/\Lie{U}(1).$ Hence, one obtains the {\em homogeneous fibration} over ${\mathbb H}P^{m}$
 \[
 \Lie{Sp}(1)/\Lie{U}(1) \to \Lie{Sp}(m+1)/(\Lie{Sp}(m)\times \Lie{U}(1)) \to \Lie{Sp}(m+1)/(\Lie{Sp}(m)\times \Lie{Sp}(1)).
 \]
 Every $\Lie{Sp}(m+1)$-invariant metric on ${\mathbb C}P^{n}=\Lie{Sp}(m+1)/(\Lie{Sp}(m)\times \Lie{U}(1))$ is determined by an inner product $\langle\cdot,\cdot\rangle$ on ${\mathfrak p}$ such that $\{Y_{\alpha}; iY_{\alpha};jY_{\alpha}; kY_{\alpha}\}_{\alpha=1}^{m}$ is an orthonormal basis of ${\mathfrak p}_{1},$ ${\mathfrak p}_{0}$ is orthogonal to ${\mathfrak p}_{1}$ and $\|X_{2}\|^{2} = \|X_{3}\|^{2} = s,$ for some $s>0.$ Then, if we suppose that $\langle\cdot, \cdot\rangle$ is ${\mathrm A}{\mathrm d}(\Lie{Sp}(m+1))$-invariant, it follows, using (\ref{tt2}) and (\ref{bracS}), that
 \[
 1 = \langle jY_{\alpha}, jY_{\alpha}\rangle = \langle [Y_{\alpha},X_{2}],jY_{\alpha}\rangle = -\langle [Y_{\alpha},jY_{\alpha}],X_{2}\rangle = 2\langle X_{2},X_{2}\rangle = 2s.
 \]
 Hence, $s = \frac{1}{2}$ and then $\langle X,Y\rangle = -\frac{1}{4}{\rm trace}\;XY,$ for all $X,Y\in {\mathfrak s}{\mathfrak p}(m+1).$ It determines a standard $\Lie{Sp}(m+1)$-homogeneous Riemannian metric with positive sectional curvature and pinching $\delta = \frac{1}{16}$ \cite{Z1} and so, it is not the symmetric Fubini-Study one. It proves the following:

 \begin{proposition}\label{mean1} A simply connected,
$2n$-dimensional, normal homogeneous space of positive
sectional curvature is isometric to a compact rank one symmetric
space: $S^{2n}$ $(\delta =1);$ ${\mathbb C}P^{n},$ ${\mathbb H}P^{n/2}$ $(n$ even), ${\mathbb C}aP^{2}$ $(n = 8),$ $(\delta =\frac{1}{4});$ or to the complex projective space
${\mathbb C}P^{n} =
\Lie{Sp}(m+1)/(\Lie{Sp}(m)\times \Lie{U}(1)),$ $n = 2m+1,$ equipped with the standard
$\Lie{Sp}(m+1)$-homogeneous Riemannian metric $(\delta = \frac{1}{16}).$
\end{proposition}
\begin{remark}{\rm The normal ${\Lie Sp}(m+1)$-homogeneous metric on ${\mathbb C}P^{n}$ is Einstein if and only if $n =3$ \cite{Z1}.
}
\end{remark}

Using the list of groups acting transitively on spheres given by Montgomery, Samelson and Borel (see \cite{Onis}), Ziller \cite{Z1} obtained all homogeneous Riemannian metrics on the sphere $S^{n}.$ Next, we find those which are normal homogeneous with positive curvature.
\begin{proposition}\label{pmean2} All normal homogeneous metrics on spheres have positive sectional curvature and they determine the following Riemannian manifolds:
\begin{enumerate}
\item[{\rm (i)}] the Euclidean sphere $S^{n};$
\item[{\rm (ii)}] $(S^{2m+1} = \Lie{SU}(m+1)/\Lie{SU}(m),g_{s});$
\item[{\rm (iii)}] $(S^{4m+3}= \Lie{Sp}(m+1)/\Lie{Sp}(m),g_{s}),$
\end{enumerate}
where in {\rm (ii)} and {\rm (iii)} $g_{s},$ $0<s\leq 1,$ denotes a normal homogeneous metric with corresponding constant pinching $\delta(s) = \frac{s(m+1)}{8m - 3s(m+1)}$ and
$$
\delta(s) = \left\{
\begin{array}{l}
 \frac{s}{8-3s}\;\;\; {\mbox if}\; s\in [\frac{2}{3},1];\\[0.6pc]
\frac{s^{2}}{4},\;\;\; {\mbox if}\; s\in ]0,\frac{2}{3}[.
\end{array}
\right.
$$
In {\rm (ii)} for $s<1$ $(s=1)$ they are $\Lie{SU}(m+1)\times \Lie{U}(1)$ $(\Lie{SU}(m+1))$-normal homogeneous and in {\rm (iii)}, ${\Lie Sp}(m+1)\times \Lie{Sp}(1)$ $(\Lie{Sp}(m+1))$-normal homogeneous.
\end{proposition}
\begin{proof} Consider the complex projective space ${\mathbb C}P^{m} = \Lie{SU}(m+1)/S(\Lie{U}(m)\times \Lie{U}(1))$ as Hermitian symmetric space, equipped with the $\Lie{SU}(m+1)$-standard homogeneous metric determined by $\langle X,Y\rangle = -\frac{1}{2}{\rm trace}\; XY$ on ${\mathfrak s}{\mathfrak u}(m+1).$ Let ${\mathfrak s}{\mathfrak u}(m+1) = {\mathfrak k} \oplus {\mathfrak m}_{1}$ be the canonical decomposition, where ${\mathfrak k}$ is the Lie algebra of the isotropy group and ${\mathfrak m}_{1} = {\mathfrak k}^{\bot}\cong {\mathbb C}^{m}.$ The center ${\mathcal Z}({\mathfrak k})$ of ${\mathfrak k}$ is one-dimensional and we have ${\mathfrak k} = {\mathfrak s}{\mathfrak u}(m)\oplus {\mathcal Z}({\mathfrak k}).$ Let $Z_{0}\in {\mathcal Z}({\mathfrak k})$ such that $\|Z_{0}\| =1.$ Then $Z_{0}$ determines a unit Killing vector field $\xi$ on $S^{2m+1} = \Lie{SU}(m+1)/\Lie{SU}(m)$ and its integral curves -they are geodesics- are the fibers of the Hopf fibration $S^{1} \to S^{2m+1} \to {\mathbb C}P^{m}.$ The natural projection is a Riemannian submersion, where $S^{2m+1}$ is considered with the $SU(m+1)$-standard metric. Then all $\Lie{SU}(m+1)$-invariant metrics on $S^{2m+1}$ are obtained, up to a scaling factor, taking $\langle\cdot,\cdot \rangle$ on ${\mathfrak m}_{1},$ $Z_{0}$ orthogonal to ${\mathfrak m}_{1}$ and $\|Z_{0}\|^{2} = s,$ for same $s>0.$ From \cite[Theorem 3]{Z}, each one of these metrics $g_{s}$ is $(\Lie{SU}(m+1)\times \Lie{U}(1))$-naturally reductive; for $s\neq 1$ it is not $\Lie{SU}(m+1)$-normal homogeneous and $g_{s}$ is $(\Lie{SU}(m+1)\times \Lie{U}(1))$-normal homogeneous if and only if $s<1.$ Then the metrics $g_{s},$ for $s\leq 1,$ are normal homogeneous. Moreover, they have positive sectional curvature. The minimum and maximum sectional curvatures are \cite{Z}: $s\frac{m+1}{2m}$ and $4 - 3s\frac{m+1}{2m}$ and it gives the corresponding pinching constants $\delta(s).$

 Every $\Lie{Sp}(m+1)$-invariant metric on $S^{4m+3}$ is determined, up to scaling factor, by an inner product $\langle\cdot,\cdot\rangle$ on ${\mathfrak m}$ such that $\{Y_{\alpha}; iY_{\alpha};jY_{\alpha}; kY_{\alpha}\}_{\alpha=1}^{n}$ is an orthonormal basis of ${\mathfrak m}_{1},$ ${\mathfrak m}_{0}$ is orthogonal to ${\mathfrak m}_{1}$ and $\|X_{p}\|^{2} = t_{p},$ for some $t_{p}>0,$ $p=1,2,3.$ From (\ref{bracS}), $\langle\cdot,\cdot\rangle$ is naturally reductive if and only if $t_{1} = t_{2} = t_{3}= \frac{1}{2}.$ Because $\Lie{Sp}(m)$ acts trivially on ${\mathfrak m}_{0}$ and $[{\mathfrak m}_{0},{\mathfrak m}_{0}]\subset {\mathfrak m}_{0},$ it follows from \cite[Theorem 3]{Z} that the $Sp(m+1)$-invariant metrics $g_{s}$ on $S^{4m+3}$ where $t_{1} = t_{2} = t_{3} =\frac{s}{2}$ are $\Lie{Sp}(m+1)\times \Lie{Sp}(1)$-naturally reductive and $\Lie{Sp}(m+1)\times \Lie{Sp}(1)$-normal homogeneous if and only if $s<1.$ So the metrics $g_{s},$ for $s\leq 1,$ are normal homogeneous. From \cite{Z1}, they have positive sectional curvature with corresponding pinching $\delta(s)$ as in (iii).

 The rest of Lie groups acting transitively on spheres and with reducible isotropy are $\Lie{Sp}(m+1)\times \Lie{U}(1)$ on $S^{4m+3} = \Lie{Sp}(m+1)\times \Lie{U}(1)/\Lie{Sp}(m)\times \Lie{U}(1)$ and $\Lie{Spin}(9)$ on $S^{15} = \Lie{Spin}(9)/\Lie{Spin}(7).$ For both cases, all invariant metrics, different from those given above, are not naturally reductive \cite{Z1}. Those with irreducible isotropy are $\Lie{SO}(n+1),$ acting on $S^{n} = \Lie{SO}(n+1)/\Lie{SO}(n),$ $\Lie{Spin}(7)$ on $S^{7} = \Lie{Spin}(7)/\Lie{G}_{2}$ and $\Lie{G}_{2}$ on $S^{6} = \Lie{G}_{2}/\Lie{SU}(3).$ In all three cases the unique invariant metric, up to a scalar factor, has constant sectional curvature \cite{B1}. Hence, it follows that all normal homogeneous metrics on spheres have positive sectional curvature and it gives the result.
 \end{proof}

 \begin{remark}{\rm $S^{2m+1}$ equipped with one of the metrics $g_{s},$ $s\leq 1,$ in (ii) is known as a {\em Berger sphere}. Such metrics are not Einstein except the one of constant curvature for $m = 1$ \cite{J1}. Since $\Lie{U}(m+1)$ also acts by isometries on $(S^{2m+1},g_{s}),$ it follows that the set of $\Lie{U}(m+1)$-invariant metrics on $S^{2m+1}$ coincides with the set of $\Lie{SU}(m+1)$-invariant metrics $g_{s}.$

Following \cite{Z1}, the unique non-Euclidean normal homogeneous Riemannian metric on sphe\-res which is Einstein is given in (iii) for $s = \frac{2}{2m+3}$ on $S^{4m+3},$ which is the Jensen's example \cite{J1}.
}
\end{remark}
\section{Existence of non-isotropic Jacobi fields}\indent

Let $(M = \Lie{G}/\Lie{K},g)$ be a normal homogeneous space and $\langle\cdot,\cdot\rangle$ its corresponding bi-invariant inner product on the Lie algebra ${\mathfrak g}$ of $\Lie{G}.$ Consider a closed subgroup $H$ of $G$ such that $K\subset H \subset G.$ Then we define the homogeneous fibration
\[
F=\Lie{H}/\Lie{K}\to M=\Lie{G}/\Lie{K}\to \tilde{M}= \Lie{G}/\Lie{H}\;\;:\;\; g\Lie{K}\mapsto g\Lie{H},
\]
with fiber $F$ and structural group $\Lie{H}.$ The $\langle\cdot,\cdot\rangle$-orthogonal decompositions ${\mathfrak h} = {\mathfrak k}\oplus {\mathfrak m}_{0},$ ${\mathfrak g} = {\mathfrak k}\oplus {\mathfrak m} = {\mathfrak k} \oplus {\mathfrak m}_{0}\oplus {\mathfrak m}_{1}$ and ${\mathfrak g} = {\mathfrak h}\oplus {\mathfrak m}_{1}$ determine reductive decompositions for $F,$ $M$ and $\tilde{M},$ respectively, and ${\mathfrak m}_{0}$ and ${\mathfrak m}_{1}$ are invariant under the linear isotropy action of $K.$ The projection $\pi:(M,g)\to (\tilde{M},\tilde{g}),$ where $\tilde{g}$ is induced by $\langle\cdot,\cdot\rangle_{{\mathfrak m}_{1}\times {\mathfrak m}_{1}},$ is a Riemannian submersion and, because ${\mathfrak m}_{0}$ is a L.t.s. of ${\mathfrak m},$ its fibers are totally geodesic.

Given a horizontal geodesic $\gamma_{u}(t) = (\exp tu)o,$ $u\in {\mathfrak m}_{1},$ of $(M,g),$ we consider the geodesic $\tilde{\gamma}_{u}(t) = (\pi\comp \gamma_{u})(t) = (\exp tu)\tilde{o}$ of $(\tilde{M},\tilde{g}),$ where $\tilde{o}$ is the origin of $\tilde{M}.$
\begin{lemma}\label{lsubmersion} If $\gamma_{u}(t_{0})$ is a $G$-isotropic conjugate point to $o$ in $M,$ then $\tilde{\gamma}_{u}(t_{0})$ is $G$-isotropic conjugate to $\tilde{o}$ in $\tilde{M}.$
\end{lemma}
\begin{proof} Let $V$ be a non-zero $G$-isotropic Jacobi field along $\gamma_{u}$ with $V(0)=0.$ Then there exists $A\in {\mathfrak k}$ such that $V = A^{*}\comp \gamma_{u}.$ Putting  $\tilde{V} = \pi_{*\gamma_{u}}V,$ one gets
\[
\tilde{V}(t) = \pi_{*\gamma_{u}(t)}\frac{d}{ds}_{\mid s =0}(\exp sA)\gamma_{u}(t) = \frac{d}{ds}_{\mid s=0}\pi((\exp sA)\gamma_{u}(t)) = \frac{d}{ds}_{\mid s=0}(\exp sA)\tilde{\gamma}_{u}(t) = A^{*}_{\tilde{\gamma}_{u}}.
\]
Hence, $\tilde{V}$ is a $G$-isotropic Jacobi field along $\tilde{\gamma}_{u},$ which using (\ref{ini}) is not identically zero and if $V(t_{0}) = 0,$ for some $t_{0},$ then $\tilde{V}(t_{0}) = 0.$
\end{proof}

Let $\mathcal{S}\subset {\mathfrak m}$ be the unit sphere of ${\mathfrak m}= {\mathfrak k}^{\bot}$ and let ${\mathfrak n}\subset {\mathfrak m}$ be an ${\mathrm A}{\mathrm d}(\Lie{K})$-invariant subspace of ${\mathfrak m}.$ Since $\langle\cdot,\cdot\rangle$ is ${\mathrm A}{\mathrm d}(\Lie{K})$-invariant, the linear isotropy action can be restricted to ${\mathcal S}\cap{\mathfrak n}.$

\begin{proposition}\label{ptrans} The following conditions are equivalent:
\begin{enumerate}
\item[{\rm (i)}] the linear isotropy action of $\Lie{K}$ on ${\mathcal S}\cap {\mathfrak n}$ is transitive;
\item[{\rm (ii)}] $[u,v]_{\mathfrak k}\neq 0$ for all linearly independent $u,v\in {\mathfrak n};$
\item[{\rm (iii)}] there exists $u\in{\mathfrak n}\setminus \{0\}$ such that $[u,v]_{\mathfrak k}\neq 0$ for all $v\in {\mathfrak n}$ linearly independent to $u.$
\end{enumerate}
\end{proposition}
\begin{proof} Because $\frac{d}{dt}_{\mid t =0}{\mathrm A}{\mathrm d}_{\exp tZ}u = \frac{d}{dt}_{\mid t = 0}(e^{t{\rm ad} Z})u = [Z,u],$ for all $Z\in {\mathfrak k}$ and $u\in {\mathfrak m},$ it follows that the tangent space $T_{u}(\Lie{K}\cdot u)$ of the orbit $\Lie{K}\cdot u = {\mathrm Ad}(\Lie{K})u$ is the bracket $[{\mathfrak k},u].$ Hence the action $\Lie{K}\times ({\mathcal S}\cap{\mathfrak n})\to {\mathcal S}\cap{\mathfrak n}$ is transitive if and only if, for some $u\in {\mathcal S}\cap{\mathfrak n},$ the orthogonal complement ${\mathfrak n}_{1}$ of $[{\mathfrak k},u]$ in ${\mathfrak n}$ is generated by $u.$ But from (\ref{tt2}) we have ${\mathfrak n}_{1} = \{v\in {\mathfrak n}\mid \langle [{\mathfrak k},u],v\rangle = 0\} = \{v\in {\mathfrak n}\mid \langle [u,v],{\mathfrak k}\rangle = 0 \} = \{v\in {\mathfrak n}\mid [u,v]_{\mathfrak k} = 0\}$ and it proves the result.
\end{proof}

\noindent As a direct consequence, using Lemma \ref{lrank}, one obtains the following well-known result.
\begin{corollary} Let $M = \Lie{G}/\Lie{K}$ be a symmetric space, $(\Lie{G},\Lie{K})$ a Riemannian symmetric pair and $\Lie{G}$ semi-simple. Then the linear isotropy action of $\Lie{K}$ on the unit sphere of ${\mathfrak m}$ is transitive if and only if ${\mathrm rank}\;M =1.$
\end{corollary}

Each unit vector $u\in {\mathfrak m}= {\mathfrak m}_{0}\oplus {\mathfrak m}_{1}$ can be written as
\begin{equation}\label{theta}
u = u(\theta) = \cos \theta u_{0} + \sin \theta u_{1},
\end{equation}
where $u_{0}\in {\mathfrak m}_{0}$ and $u_{1}\in {\mathfrak m}_{1},$ $\|u_{0}\|= \|u_{1}\| =1,$ and $\theta$ is the slope angle $\theta = {\mathrm ang}(u,{\mathfrak m}_{0})\in [0,\pi/2].$ From Lemma \ref{isotropic}, if the linear isotropy action of $\Lie{K}$ on ${\mathcal S}\cap {\mathfrak m}_{1}$ is transitive the study of conjugate points along geodesics on normal homogeneous spaces $(M = G/K,g)$ can be reduced to consider geodesics $\gamma_{u}$ where $u$ is given as in (\ref{theta}) but fixing an arbitrary unit vector $u_{1}\in {\mathfrak m}_{1}.$ Using this fact and Theorem \ref{cp}, together with Lemma \ref{lsubmersion}, we shall find isotropic and non-isotropic conjugate points to the origin in geodesics starting at that point.
\vspace{0.2cm}
\subsection{$(S^{2m+1},g_{\kappa,s}),$  $(S^{4m+3},g_{\kappa,s}),$  $(s<1),$ {\sc and} $({\mathbb C}P^{2m+1},g_{\kappa})$}\indent

\vspace{0.2cm}

Denote by $S^{m}(\kappa)$ the sphere of radius $\frac{1}{\sqrt{\kappa}},$ or equivalently, of constant curvature $\kappa>0,$ by $\RR P^{m}(\kappa)$ the real projective space of constant curvature $\kappa,$ by ${\mathbb C}P^{m}(\kappa)$ the complex projective space with constant holomorphic sectional curvature $c=4\kappa$ and by ${\mathbb H}P^{m}(\kappa)$ the quaternionic projective space with constant quaternionic sectional curvature $c = 4\kappa$ (the sectional curvature varies between $\kappa$ and $4\kappa).$

In order to study conjugate points along geodesics in $(S^{2m+1},g_{\kappa,s} = \frac{1}{\kappa}g_{s}),$ $(S^{4m+3},g_{\kappa,s}= \frac{1}{\kappa}g_{s})$ and $({\mathbb C}P^{2m+1},g_{\kappa} = \frac{1}{\kappa}g),$ we first need to determine the reductive decompositions ${\mathfrak g} = {\mathfrak k}\oplus {\mathfrak m}_{s}$ (they only depend on the parameter $s)$ associated to the corresponding normal homogeneous quotients. For $s = 1,$ including in this case $({\mathbb C}P^{2m+1},g_{\kappa}),$ such reductive decompositions have already been obtained in the previous section.

We start considering Berger spheres $(S^{2m+1} = \Lie{SU}(m+1)/\Lie{SU}(m),g_{s}),$ $s<1.$ Put $S_{j} = \frac{1}{\alpha_{j}}\sum_{l=1}^{j}lA_{l,l+1},$ $j=1,\dots ,m,$ where $\alpha_{j} = (\frac{j(j+1)}{2})^{1/2}.$ Then an orthonormal basis of ${\mathfrak s}{\mathfrak u}(m+1)$ with respect to the inner product $<X,Y> = -\frac{1}{2}{\rm trace}\; XY,$ is given by $\{A_{l,l+1},\;l=1,\dots ,m;\; B_{rj},\; C_{rj};\; 1\leq r< j\leq m+1\}$  and $\{S_{j}, j = 1,\dots ,m-1;\; B_{rj},\; C_{rj};\; 1\leq r< j\leq m\}$ is an orthonormal basis of ${\mathfrak s}{\mathfrak u}(m)$ embedded in the usual way in ${\mathfrak s}{\mathfrak u}(m+1).$ Moreover, up to sign, $Z_{0}= S_{m},$ where $Z_{0}$ is defined as in proof of Proposition \ref{pmean2}. An orthonormal basis for ${\mathfrak m}_{1}\cong {\mathbb C}^{m}$ is given by
\[
e_{r} = B_{r,m+1},\;\;\; f_{r} = C_{r,m+1},\;\;r= 1,\dots ,m.
\]
Let $D$ be a basis element of the Lie algebra $\RR$ of $S^{1}\cong \Lie{U}(1)$ and consider the coset space $G/K,$ where $\Lie{G} = \Lie{SU}(m+1)\times \Lie{U}(1)$ and $\Lie{K}$ is the connected Lie subgroup of $\Lie{G}$ with Lie algebra ${\mathfrak k}= {\mathfrak s}{\mathfrak u}(m)\oplus \RR$ generated by
\[
\{h_{s} = \sqrt{1-s}(Z_{0} + D), S_{1},\dots ,S_{m-1}; B_{rj}, C_{rj}, 1\leq r<j\leq m\}.
\]
Then ${\mathfrak g} = {\mathfrak k} \oplus {\mathfrak m}_{s}$ is a reductive decomposition for $\Lie{G}/\Lie{K}$ being ${\mathfrak m}_{s} = {{\mathfrak m}_{0}}_{s}\oplus {\mathfrak m}_{1},$ ${{\mathfrak m}_{0}}_{s} = \RR d_{s}$ and
\[
d_{s} = \sqrt{s}(Z_{0} + \frac{s-1}{s}D).
\]
Using (\ref{bracABC}), the Lie brackets $[{\mathfrak m}_{s},{\mathfrak m}_{s}]$ are determined by
\begin{equation}\label{brac1}
\begin{array}{l}
[d_{s},e_{r}] = \sqrt{s}\frac{m+1}{\alpha_{m}}f_{r},\;\;\; [d_{s},f_{r}] = -\sqrt{s}\frac{m+1}{\alpha_{m}}e_{r},\\[0.4pc]
[e_{r},e_{j}] = [f_{r},f_{j}] = -B_{rj},\;\;[e_{r},f_{j}] = C_{rj},\;\;\;\mbox{if}\; r\neq j,\\[0.4pc]
[e_{r},f_{r}] = \frac{m+1}{\alpha_{m}}(\sqrt{s}d_{s} + \sqrt{(1-s)}h_{s}) + \frac{1-r}{\alpha_{r-1}}S_{r-1} + \frac{1}{\alpha_{r}}S_{r} + \dots + \frac{1}{\alpha_{m-1}}S_{m-1}.
\end{array}
\end{equation}
Moreover, the inner product on ${\mathfrak s}{\mathfrak u}(m +1)\oplus \RR$ making
\[
\{h_{s},S_{1},\dots ,S_{n-1};B_{rj},C_{rj},\;1\leq r<j\leq m;\; d_{s},e_{r},f_{r},\; 1\leq r\leq m\}
\]
an orthonormal basis determines an bi-invariant metric on $SU(m+1)\times \Lie{U}(1)$ and makes $S^{2m+1}$ in a $SU(m+1)\times \Lie{U}(1)$-normal homogeneous space isometric to $(S^{2m+1},g_{s})$ \cite{Z}. Since $[d_{s},{\mathfrak k}] = 0,$ $d_{s}$ is ${\mathrm A}{\mathrm d}(\Lie{S}(\Lie{U}(m)\times \Lie{U}(1)))$-invariant on ${\mathfrak m}_{s}$ and so, it determines a $G$-invariant unit vector field $\xi$ on $S^{2m+1},$ the {\em Hopf vector field}. Because $\xi$ is a Killing vector field, every geodesic $\gamma_{u}$ on $(S^{2m+1},g_{s})$ intersects each fiber with a constant slope angle $\theta = {\mathrm ang}(\xi_{o},u)\in ]0,\pi[.$

For $(S^{4m+3}= \Lie{Sp}(m+1)/\Lie{Sp}(m),g_{s}),$ $s< 1,$ let $\{D_{p}\}_{p=1,2,3}$ be the standard basis of ${\mathfrak s}{\mathfrak p}(1)\cong {\mathfrak s}{\mathfrak u}(2)$ given by $D_{1} = A_{1,2},$ $D_{2} = B_{1,2}$ and $D_{3} = C_{1,2}.$ Put
\begin{equation}\label{dh}
{d_{p}}_{s} = \sqrt{2s}(X_{p} + \frac{s-1}{s}D_{p}),\;\;\;\; {h_{p}}_{s} = \sqrt{2(1-s)}(X_{p} + D_{p}).
\end{equation}
Then ${\mathfrak s}{\mathfrak p}(m+1)\oplus {\mathfrak s}{\mathfrak p}(1) = {\mathfrak k}\oplus {\mathfrak m}_{s}$ is a reductive decomposition associated to the quotient space $\Lie{Sp}(m+1)\times\Lie{Sp}(1)/\Lie{K},$ where $\Lie{K}$ is the connected Lie subgroup of $\Lie{Sp}(m+1)\times\Lie{Sp}(1)$ with Lie algebra ${\mathfrak k} = {\mathfrak s}{\mathfrak p}(m)\oplus {\mathfrak s}{\mathfrak p}(1)$ generated by $\{Z_{\alpha,\beta},Z_{\alpha p},Z_{(\alpha,\beta) p};{h_{p}}_{s}\},$ for $p = 1,2,3$ and $1\leq \alpha <\beta \leq m,$ and ${\mathfrak m}_{s}= {{\mathfrak m}_{0}}_{s}\oplus {\mathfrak m}_{1}$ the vector space with adapted basis $\{{d_{p}}_{s};\;Y_{\alpha},Y_{\alpha p}\}.$ Using (\ref{bracS}) and (\ref{bracS1}), the inner product on ${\mathfrak s}{\mathfrak p}(m+1)\oplus {\mathfrak s}{\mathfrak p}(1)$ making these vectors an orthonormal basis is bi-invariant and, in a similar way than for Berger spheres, $S^{4m+3}$ becomes into a $\Lie{Sp}(m+1)\times \Lie{Sp}(1)$-normal homogeneous space isometric to $(S^{4m+3},g_{s}).$

For each $\kappa >0,$ using the well-known O'Neill formula (\cite[Ch. 9]{B}) and (\ref{curnor}), (\ref{bracS}) and (\ref{brac1}), we have the following homogeneous Riemannian fibrations:
\begin{enumerate}
\item[{\rm (i)}] $S^{1}\to (S^{2m+1}, g_{\kappa ,s} = \frac{1}{\kappa}g_{s}) \to {\mathbb C}P^{m}(\kappa);$
\item[{\rm (ii)}] $S^{2}\to ({\mathbb C}P^{2m+1}, g_{\kappa} = \frac{1}{\kappa}g)\to {\mathbb H}P^{m}(\kappa);$
\item[{\rm (iii)}] $S^{3}\to (S^{4m+3},g_{\kappa,s} = \frac{1}{\kappa}g_{s}) \to {\mathbb H}P^{m}(\kappa).$
\end{enumerate}
\begin{remark}\label{rsymmetric}{\rm On compact symmetric spaces, the conjugate points to the origin along any geodesic $\gamma_{u}$ are strictly isotropic and they are given by $\gamma_{u}(\frac{p\pi}{\sqrt{\lambda}}),$ $p\in {\mathbb Z},$ where $\lambda$ is a eigenvalue of $R_{u}.$ In particular, on $S^{n}(\kappa)$ or $\RR P^{n}(\kappa)$ the conjugate points to the origin are all $\gamma_{u}(\frac{p\pi}{\sqrt{\kappa}})$ and $\gamma_{u}(\frac{p\pi}{2\sqrt{\kappa}})$ on ${\mathbb K}P^{n}(\kappa),$ for ${\mathbb K} = {\mathbb C},$ ${\mathbb H}$ or ${\mathbb C}a.$
}
\end{remark}

Now, let ${\mathfrak g} = {\mathfrak k}\oplus {\mathfrak m}_{s},$ $s\leq 1,$ be any of the reductive decompositions previously obtained and let ${\mathfrak m}_{s} = {{\mathfrak m}_{0}}_{s}\oplus {\mathfrak m}_{1}$ be the ${\mathrm A}{\mathrm d}(\Lie{K})$-invariant orthogonal decomposition of ${\mathfrak m}_{s},$ where ${{\mathfrak m}_{0}}_{s}$ is the L.t.s. subspace of ${\mathfrak m}_{s}$ associated to the fibers as totally geodesic submanifolds.
\begin{lemma}\label{ltran} For $(S^{2m+1},g_{\kappa,s}),$ $({\mathbb C}P^{2m+1},g_{\kappa})$ and $(S^{4m+3},g_{\kappa,s}),$ the linear isotropy action on ${\mathcal S}\cap {\mathfrak m}_{1}$ is transitive except for the Euclidean sphere $(S^{3},g_{\kappa,1}).$ Moreover, it is transitive on ${\mathcal S}\cap {{\mathfrak m}_{0}}_{s}$ except for $(S^{4m+3},g_{\kappa,1}).$
\end{lemma}
\begin{proof} Given $v = \sum_{r=1}^{m}(v_{1}^{r}e_{r} + v_{2}^{r}f_{r})\in {\mathfrak m}_{1}\cong {\mathbb C}^{m},$ (\ref{brac1}) implies that
\[
[e_{m},v]_{{\mathfrak s}{\mathfrak u}(m)} = \sum_{r=1}^{m-1}(v^{r}_{1}B_{rm} + v^{r}_{2}C_{rm}) - v^{m}_{2}\frac{m-1}{\alpha_{m-1}}S_{m-1}
\]
and, putting $v = \sum_{\alpha =1\atop p = 1,2,3}^{m}(v^{\alpha}Y_{\alpha} + v^{\alpha p}Y_{\alpha p})\in {\mathfrak m}_{1}\cong {\mathbb H}^{m},$ it follows from (\ref{bracS}), (\ref{bracS1})
\[
[Y_{1},v]_{{\mathfrak s}{\mathfrak p}(m)} = \sum_{\alpha =2}^{m}(-v^{\alpha}Z_{1,\alpha} + \sum_{p=1,2,3}v^{\alpha p}Z_{(1,\alpha)p}) + 2\sum_{p=1,2,3}v^{1 p}Z_{1 p}.
\]
Then $[e_{m},v]_{{\mathfrak s}{\mathfrak u}(m)},$ for $m>1,$ or $[Y_{1},v]_{{\mathfrak s}{\mathfrak p}(m)}$ is zero if and only if $u$ is collinear to $e_{m}$ or $v$ to $Y_{1}.$ For $(S^{3},g_{\kappa,s}),$ we obtain $[e_{1},v]_{\mathfrak k} = 2\sqrt{1-s}v^{1}_{2}h_{s}.$ Hence, using Proposition \ref{ptrans}, the linear isotropy action on ${\mathcal S}\cap{\mathfrak m}_{1}$ is transitive in each case. Also the linear isotropy action on ${\mathcal S}\cap {{\mathfrak m}_{0}}_{s}$ is transitive for $(S^{2m+1},g_{\kappa,s})$ and $({\mathbb C}P^{2m+1},g_{\kappa}).$ On $(S^{4m+3},g_{\kappa,s}),$ taking $x_{s} = \sum_{p=1}^{3}x_{s}^{p}{d_{p}}_{s}$ one gets
\[
[{d_{1}}_{s}, x_{s}]_{\mathfrak k} = 2\sqrt{2(1-s)}(x^{2}_{s}{h_{3}}_{s} - x^{3}_{s}{h_{2}}_{s}).
\]
Hence, using again Proposition \ref{ptrans}, the last part of the lemma is proved.
\end{proof}

\noindent Using (\ref{curnor}) and (\ref{bracS}) and (\ref{brac1}), it follows
\begin{lemma}\label{tau} The sectional curvature $K(u_{0},u_{1}),$ for all $u_{0}\in {{\mathfrak m}_{0}}_{s}$ and $u_{1}\in {\mathfrak m}_{1},$ is a function $\tau= \tau(\kappa, s) = \frac{\kappa s(m+1)}{2m}$ on $(S^{2m+1}, g_{\kappa ,s}),$ $\tau = \frac{\kappa}{2}$ on $({\mathbb C}P^{2m+1}, g_{\kappa})$ and $\tau = \frac{\kappa s}{2}$ on $(S^{4m+3},g_{\kappa,s}).$
\end{lemma}
\begin{remark}{\rm $\tau<\kappa$ except in the $3$-dimensional Euclidean sphere $(S^{3},g_{\kappa,1}),$ where $\tau = \kappa.$}
\end{remark}
Next, we shall prove the following result, which generalizes one given in \cite{Ch2} for Berger spheres.

\begin{theorem}\label{conj} On $(S^{2m+1},g_{\kappa,s})$ and $(S^{4m+3},g_{\kappa,s}),$ for $s\leq 1,$ and on $({\mathbb C}P^{2m+1},g_{\kappa})$ any geodesic $\gamma_{u},$ $u = u(\theta),$ admits conjugate points which are not strictly isotropic and there exist geodesics admitting non-isotropic conjugate points. Concretely, we have:
\begin{enumerate}
\item[{\rm (i)}] If $\gamma_{u}$ is a vertical geodesic, i.e. $\theta =0,$ the points $\gamma_{u}(\frac{p\pi}{\sqrt{\tau}}),$ $p\in {\mathbb Z},$ are conjugate points but not strictly isotropic. On $(S^{2m+1},g_{\kappa,s})$ these are all conjugate points along the Hopf fiber $\gamma_{u}$ and they are not isotropic.
\item[{\rm (ii)}] If $\theta\in ]0,\frac{\pi}{2}],$ the points of the form $\gamma_{u}(\frac{t}{2\sqrt{\kappa\sin^{2}\theta + \tau \cos^{2}\theta}}),$ where
\begin{enumerate}
\item[{\rm (A)}] $t$ is a solution of the equation $\tan \frac{t}{2} = \frac{t(\tau -\kappa)\sin^{2}\theta}{2\tau},$
\noindent or
\item[{\rm (B)}] $t = 2p\pi,$ $p\in {\mathbb Z},$
\end{enumerate}
are conjugate points to the origin. In the first case, they are not strictly isotropic and in the second one, they are isotropic.
\item[{\rm (iii)}] If $\gamma_{u}$ is a horizontal geodesic, $\theta=\frac{\pi}{2},$ the points $\gamma_{u}(\frac{t}{2\sqrt{\kappa}})$ as in {\rm (ii) (A)}, where $t$ is a solution of the equation $\tan \frac{t}{2} = \frac{t(\tau -\kappa)}{2\tau},$ are conjugate points which are not isotropic.
\end{enumerate}
\end{theorem}
\begin{proof} We shall show that in each one of these spaces there exist orthonormal vectors $u,v\in {\mathfrak m}_{s}$ satisfying the conditions of Theorem \ref{cp} (ii) for coefficients $\lambda$ and $\rho$ given by $\lambda = 4\tau$ and $\rho = 4(\kappa - \tau)\sin^{2}\theta:$

On Berger spheres $(S^{2m+1},g_{\kappa,s}),$ by using of Lemma \ref{ltran}, we restrict our study to geodesics $\gamma_{u}$ starting at the origin with $u = \sqrt{\kappa}(\cos \theta d_{s} + \sin\theta e_{m}).$ Put $v = \sqrt{\kappa}(\cos \theta e_{m} - \sin\theta d_{s}).$ Then, using (\ref{brac1}) and Lemma \ref{tau}, we get
$$
\begin{array}{l}
[u,v] = \kappa \sqrt{s}\frac{m+1}{\alpha_{m}}f_{m},\;\;\; [[u,v],u]_{{\mathfrak m}_{s}} = 2\kappa s\frac{m+1}{m}v= 4\tau v,\\[0.4pc] [u,[u,v]]_{\mathfrak k} = \kappa\sqrt{\kappa s}\frac{m+1}{\alpha_{m}}\sin\theta\Big( \frac{m+1}{\alpha_{m}}\sqrt{1-s}h_{s} + \frac{1-m}{\alpha_{m-1}}S_{m-1}\Big).
\end{array}
$$
Moreover, from (\ref{tt2}) and (\ref{brac1}), $[[u,[u,v]]_{\mathfrak k},u]$ is collinear to $f_{m}.$ It proves the result for this case. On $({\mathbb C}P^{2m+1},g_{\kappa}),$ we only need to consider geodesics $\gamma_{u}$ with $u$ given by $u = \sqrt{\kappa}(\cos\theta X + \sin\theta Y_{\alpha}),$ for some $\alpha\in \{1,\dots, m\},$ where $X = X(\phi) = \sqrt{2}(\cos\phi X_{2} + \sin \phi X_{3})\in {\mathfrak p}_{0}= {\mathfrak m}_{0}.$ Put $v = \sqrt{\kappa}(\cos\theta Y_{\alpha} - \sin \theta X)$ in Theorem \ref{cp}. Then, using (\ref{bracS}) and (\ref{bracS1}), we get
$$
\begin{array}{lcl}
[u,v] & = & -\sqrt{2}\kappa(\cos\phi Y_{\alpha 2} + \sin\phi Y_{\alpha 3}),\\[0.4pc]
[[u,v],u] & = & 2\kappa(v + \sqrt{2\kappa}\sin\theta(\cos\phi Z_{\alpha 2} + \sin\phi Z_{\alpha 3}))
\end{array}
$$
and $[[[u,v],u]_{\mathfrak k},u]$ is collinear to $[u,v].$ Finally, on $(S^{4m+3},g_{\kappa,s}),$ using again Lemma \ref{ltran}, we can take geodesics $\gamma_{u}$ with $u = \sqrt{\kappa}(\cos\theta X_{s} + \sin\theta Y_{\alpha}),$ for some $\alpha\in \{1,\dots , m\},$ where $X_{s}$ is an arbitrary vector of ${{\mathfrak m}_{0}}_{s},$ written as $X_{s} = X_{s}(\phi_{1},\phi_{2}) = \sin\phi_{1}\cos\phi_{2} {d_{1}}_{s} + \sin\phi_{1}\sin\phi_{2} {d_{2}}_{s} + \cos\phi_{2} {d_{3}}_{s}.$ Now, put $v = \sqrt{\kappa}(\cos\theta Y_{\alpha} - \sin\theta X_{s}).$ Then, from (\ref{bracS}), (\ref{bracS1}) and (\ref{dh}), taking into account Lemma \ref{tau}, one gets
$$
\begin{array}{lcl}
[u,v] & = & -\sqrt{2s}\kappa(\sin\phi_{1}\cos\phi_{2}Y_{\alpha 1} + \sin\phi_{1}\sin\phi_{2}Y_{\alpha 2} + \cos\phi Y_{\alpha 3}),\\[0.4pc] [[u,v],u]_{{\mathfrak m}_{s}} & = & 2\kappa sv = 4\tau v,\\[0.4pc]
[u,[u,v]]_{\mathfrak k} & = & 2\kappa\sqrt{\kappa s}\sin\theta\Big ((\sqrt{1-s}{h_{1}}_{s} - \sqrt{2} Z_{\alpha 1}) \sin\phi_{1}\cos\phi_{2}\\[0.3pc]
  & & + (\sqrt{1-s}{h_{2}}_{s} - \sqrt{2} Z_{\alpha 2}) \sin\phi_{1}\sin\phi_{2} + (\sqrt{1-s}{h_{3}}_{s} - \sqrt{2} Z_{\alpha 3}) \cos\phi_{2}\Big ),\\[0.4pc]
 [[u,[u,v]]_{\mathfrak k},u] & = & 2(2-s)\kappa\sin^{2}\theta [u,v].
  \end{array}
$$
Hence, the result also holds for this last case.

If $\gamma_{u}$ is vertical, $\rho =0$ and (i) follows from Theorem \ref{cp} (ii)(A). On $(S^{2m+1},g_{\kappa,s}),$ the vector $d_{s}\in {\mathfrak m}_{s}$ determines the Hopf vector field, which is $(\Lie{S}\Lie{U}(m+1)\times \Lie{U}(1))$-invariant. From (\ref{torsion}), (\ref{RR}) and (\ref{brac1}), one gets that $\{e_{r},f_{r}\}$ generates the eigenspace of the Jacobi operator $R_{d_{s}}$ with eigenvector $\frac{\tau}{\kappa}.$ Then, using \cite[Theorem 5.3]{GS}, the points $\gamma_{u}(\frac{p\pi}{\sqrt{\tau}})$ are all conjugate points along the Hopf fibers, their multiplicity is $2m$ and they are not isotropically conjugate. For (ii) we use directly Theorem \ref{cp} (ii) (B). Finally, (iii) follows from (ii) and Lemma \ref{lsubmersion} and Remark \ref{rsymmetric}.
\end{proof}

\begin{corollary}\label{cimp1} Any vertical geodesic in $({\mathbb C}P^{2m+1},g_{\kappa})$ admits isotropic conjugate points which are not strictly isotropic.
\end{corollary}
\begin{proof} Put $u = \sqrt{2\kappa}(\cos\phi X_{2} + \sin\phi X_{3})$ and $v = \sqrt{2\kappa}(\cos\phi X_{3} - \sin\phi X_{2})$ in Theorem \ref{cp}. Then, from (\ref{bracS}), one gets $[u,v] = 4\kappa X_{1}$ and $[[u,v],v]= 8\kappa v.$ So, the points $\gamma_{u}(\frac{\sqrt{2\kappa}p\pi}{4\kappa})$ are isotropically conjugate to the origin. Hence, using Theorem \ref{conj} (i), the points $\gamma_{u}(\frac{\sqrt{2\kappa}p\pi}{\kappa})$ satisfy the conditions of this corollary.
\end{proof}
\begin{remark}{\rm From (\ref{curnor}) and (\ref{bracS}), the $2$-sphere $S^{2}= \Lie{Sp}(1)/\Lie{U}(1),$ isometrically embedded as totally geodesic submanifold in $({\mathbb C}P^{2m+1},g_{\kappa}),$ is the Euclidean sphere with constant curvature $8\kappa.$ Then the vertical geodesics of $({\mathbb C}P^{2m+1},g_{\kappa})$ are closed with length $\frac{\sqrt{2\kappa}\pi}{2\kappa}.$}
\end{remark}

\subsection{\sc The Berger space $B^{13} = SU(5)/H$}\indent

\vspace{0.1cm}

The isotropy subgroup $H$ is given by
\[
H:= \left\{\left(
\begin{array}{cc}
zA & 0\\
0 & {\bar{z}}^{4}
\end{array}
\right ) \mid \; A\in Sp(2)\subset SU(4),\; z\in S^{1}\subset
{\mathbb C}\right \}\subset S(U(4)\times U(1))\subset SU(5).
\]
Then $H$ may be considered as the image of $Sp(2)\times S^{1}$ under the homomorphism
\[
\chi : \Lie{SU}(4)\times S^{1}\to \Lie{U}(4)\subset \Lie{SU}(5),\;\;\; (A,z)\mapsto zA,
\]
where $\Lie{U}(4)$ is embedded in $\Lie{SU}(5)$ through the map {\footnotesize $\iota: A\to \left (\begin{array}{cc} A& 0\\ 0 & {\rm det}A^{-1}\end{array}\right )$ } or equivalently, as $(Sp(2)\times S^{1})/\{\pm(Id,1)\}.$ So, its Lie algebra ${\mathfrak h}$ is given by ${\mathfrak h} = {\mathfrak s}{\mathfrak p}(2)\oplus \RR Z_{0},$ where $Z_{0}= S_{4}.$ On ${\mathfrak s}{\mathfrak u}(5)$ we take the bi-invariant inner product $\langle X,Y\rangle = -\frac{1}{4}{\rm trace}\;XY.$ Using the fact that $Sp(2)/{\mathbb Z}_{2}\cong SO(5)$ and $SU(4)/{\mathbb Z}_{2}\cong SO(6),$ it follows that the five-dimensional Euclidean sphere $S^{5} = SO(6)/SO(5)$ can be also described as the quotient $SU(4)/Sp(2),$ with reductive orthogonal decomposition ${\mathfrak s}{\mathfrak u}(4) = {\mathfrak s}{\mathfrak p}(2)\oplus {\mathfrak m}_{0},$ ${\mathfrak m}_{0}\cong \RR^{5}.$ Moreover, because $[Z_{0},{\mathfrak m}_{0}] = 0,$ one gets that $[{\mathfrak h}, {\mathfrak m}_{0}] = [{\mathfrak s}{\mathfrak p}(2), {\mathfrak m}_{0}] \subset {\mathfrak m}_{0}.$

On the other hand, the quotient expression $SU(5)/S(U(4)\times U(1))$ for ${\mathbb C}P^{4}$ determines another reductive decomposition ${\mathfrak s}{\mathfrak u}(5) = ({\mathfrak s}{\mathfrak u}(4) \oplus \RR Z_{0})\oplus {\mathfrak m}_{1},$ ${\mathfrak m}_{1}\cong {\mathbb C}^{4}.$ Then,
$[{\mathfrak h},{\mathfrak m}_{1}]\subset [{\mathfrak s}{\mathfrak u}(4)\oplus \RR Z_{0},{\mathfrak m}_{1}] \subset {\mathfrak m}_{1}.$ Hence, ${\mathfrak s}{\mathfrak u}(5) = {\mathfrak h}\oplus {\mathfrak m},$ where ${\mathfrak m} = {\mathfrak m}_{0} \oplus {\mathfrak m}_{1}$ is a reductive decomposition of ${\mathfrak s}{\mathfrak u}(5)$ associated to $SU(5)/H.$ Since $H$ is connected, ${\mathfrak m}_{0}$ and ${\mathfrak m}_{1}$ are $Ad(H)$-invariant subspaces of ${\mathfrak m}$ and, because ${\mathfrak h}$ contains to ${\mathfrak s}{\mathfrak p}(2),$ ${\mathfrak m}_{0}$ is  isotropy-irreducible. Moreover, the corresponding totally geodesic submanifold $M_{{\mathfrak m}_{0}}$ through the origin in $B^{13}$ is given by $SU(4)/(SU(4)\cap H),$ which is diffeomorphic to the $5$-dimensional real projective space $\RR P^{5}$ because $\chi(A,1)=\chi(-A,-1),$ for all $A\in Sp(2).$
\begin{proposition}\label{proB13} The projection $\pi: B^{13}\to {\mathbb C}P^{4} = SU(5)/S(U(4)\times U(1)),$ $gH\mapsto gS(U(4)\times U(1)),$ determines a homogeneous fibration $\RR P^{5}\to B^{13}\to {\mathbb C}P^{4},$ where the fibers are obtained by the action of $SU(5)$ on $\RR P^{5} = SU(4)/(SU(4)\cap H).$
\end{proposition}
\begin{proof}
Because $S(U(4)\times U(1))$ is closed in $SU(5)$ and $H$ compact, it follows that the quotient manifold $S(U(4)\times U(1))/H$ is a regular submanifold of $B^{13}$ (see \cite[Proposition 4.4, Ch. II]{He}) and ${\mathfrak s}{\mathfrak u}(4)\oplus {\mathfrak u}(1) = ({\mathfrak s}{\mathfrak p}(2)\oplus \RR Z_{0})\oplus {\mathfrak m}_{1}$ is a reductive decomposition. Moreover, $SU(4)/(SU(4)\cap H)$ is a submanifold of $S(U(4)\times U(1))/H,$ via the immersion $g\;(SU(4)\cap H)\mapsto \iota
(g)H,$ for $g\in SU(4)$,  where $\iota \, : \,
SU(4) \to S(U(4)\times U(1))$ is the standard inclusion map. Hence, $SU(4)/(SU(4)\cap H)$ becomes into an open
submanifold of $S(U(4)\times U(1))/H.$ Taking into account that it is also closed
in $S(U(4)\times U(1))/H,$ one gets
\[
\RR P^{5} = \frac{SU(4)}{SU(4)\cap H} = \frac{S(U(4)\times U(1))}{H}.
\]
It proves the proposition.
\end{proof}

Following \cite{Hei}, an orthonormal basis of
${\mathfrak s}{\mathfrak u}(5)$ with respect to $\langle\cdot , \cdot \rangle$ and adapted to the reductive
decomposition ${\mathfrak s}{\mathfrak u}(5) = {\mathfrak h}\oplus {\mathfrak m}_{0}\oplus {\mathfrak m}_{1}$ is
\[
\{H_{1},\dots ,H_{11};u_{0},u_{1},u_{2},v_{1},v_{2}; e_{1},\dots ,e_{4},f_{1},\dots,f_{4}\}
\]
where $\{H_{1},\dots, H_{11}\}$ is the basis of ${\mathfrak h}$ given by
$$
\begin{array}{lclclclclcl}
H_{1} & = & A_{1,2} + 2A_{2,3} + A_{3,4}, & & H_{2} & = & B_{1,3}
+ B_{2,4}, & & H_{3} \hspace{-0.02cm}& \hspace{-0.02cm}= \hspace{-0.02cm}&\hspace{-0.02cm}C_{1,3} + C_{2,4},\\
H_{4} & = & A_{1,2} - A_{3,4}, & & H_{5} & = & B_{1,3} -
B_{2,4}, & & H_{6}\hspace{-0.02cm} & \hspace{-0.02cm}=\hspace{-0.02cm} &\hspace{-0.02cm} C_{1,3} - C_{2,4},\\
H_{7} & = & C_{1,2} - C_{3,4}, & & H_{8}& = & B_{1,4} + B_{2,3}, & &
H_{9}\hspace{-0.02cm}& \hspace{-0.02cm}=\hspace{-0.02cm} & \hspace{-0.02cm}C_{1,4} + C_{2,3},\\
H_{10} & = & B_{1,2} + B_{3,4}, & & H_{11}& = & \sqrt{2}Z_{0},
\end{array}
$$
the matrices $A_{jk},$ $B_{jk},$ $C_{jk},$ $1\leq j<k\leq 5,$ are defined in (\ref{ABC}),
\[
u_{0}  =  A_{1,2} + A_{3,4},\;\;\;\; u_{1}  =  B_{1,2} - B_{3,4},\;\;\; u_{2}  =  B_{1,4} - B_{2,3},\;\;\;\; v_{1}  =  C_{1,2} + C_{3,4}, \;\;\; v_{2} = C_{1,4} - C_{2,3}
\]
constitute an orthonormal basis for ${\mathfrak m}_{0}$ and
\[
e_{r} = \sqrt{2}B_{r,5},\;\;\;\; f_{r} = J_{o}e_{r} = \sqrt{2}C_{r,5},\;\;\; r = 1,\dots ,4,
\]
is a basis of ${\mathfrak m}_{1}\cong {\mathbb C}^{4},$ being $J_{o}$ the canonical complex structure $J_{o} = \frac{1}{\sqrt{5}}{\mathrm ad}_{Z_{0}}$ induced in ${\mathbb C}P^{4}.$

Every $SU(5)$-invariant metric on $B^{13}$ is determined, up to a scaling factor, by an $Ad(H)$-invariant inner product $\langle\cdot,\cdot\rangle_{s}$ on ${\mathfrak m},$ for some $s>0,$ given by \cite{Put}
\[
\langle\cdot,\cdot\rangle_{s} = s\langle\cdot,\cdot\rangle_{{\mathfrak m}_{0}\times{\mathfrak m}_{0}} + \langle\cdot,\cdot\rangle_{{\mathfrak m}_{1}\times{\mathfrak m}_{1}}.
\]

\begin{lemma}\label{lB13} The standard metric $\langle\cdot,\cdot \rangle$ is, up to homotheties, the unique $SU(5)$-invariant metric on $B^{13}$ which is normal homogeneous.
\end{lemma}
\begin{proof}
We shall show that $s$ must be $1.$ Suppose that $\langle\cdot,\cdot\rangle_{s}$ can be extended to a bi-invariant inner product of ${\mathfrak s}{\mathfrak u}(5),$ which we also denote by $\langle\cdot,\cdot\rangle_{s},$ making ${\mathfrak h}$ and ${\mathfrak m}$ orthogonal. Then, from (\ref{tt2}) and (\ref{bracABC}), we have
 \[
 1 = \langle e_{1}, e_{1}\rangle_{s} = \langle [u_{2},e_{4}],e_{1}\rangle_{s} = \langle [e_{4},e_{1}],u_{2}\rangle_{s} = \langle u_{2} + H_{8},u_{2}\rangle_{s} = \langle u_{2}, u_{2}\rangle_{s} = s.
 \]
\end{proof}

Using (\ref{bracABC}), one gets $\|[e_{r},f_{r}]_{\mathfrak h}\|^{2} = 7$ and $\|[e_{r},f_{r}]_{\mathfrak m}\|^{2} = \|[e_{r},f_{r}]_{{\mathfrak m}_{0}}\|^{2} = 1,$ for $r=1,\dots ,4.$ Then, from (\ref{curnor}) and the O'Neill formula, it follows
\[
K_{B^{13}}(e_{r},f_{r}) = \frac{29}{4},\;\;\;\; K_{{\mathbb C}P^{4}}(e_{r},f_{r}) = 8.
\]
Hence, the base space of the Riemannian submersion in Proposition \ref{proB13} is isometric to ${\mathbb C}P^{4}(2).$ Moreover, since for orthonormal vectors $u,v\in {\mathfrak m}_{0},$ one gets $\|[u,v]\|^{2} = \|[u,v]_{\mathfrak h}\|^{2} = 4,$ (\ref{curnor}) implies that the fibers are isometric to $\RR P^{5}(4).$
\begin{lemma}\label{ltrans2} The linear isotropy action of $H$ on ${\mathcal S}\cap {\mathfrak m}_{0}$ and on ${\mathcal S}\cap {\mathfrak m}_{1}$ is transitive.
\end{lemma}
\begin{proof} Given $x= x^{0}u_{0} + \sum_{i=1,2}(x^{i}_{1}u_{i} + x^{i}_{2}v_{i})\in {\mathfrak m}_{0}$ and $v = \sum_{r=1}^{4}(v^{r}_{1}e_{r} + v^{r}_{2}f_{r})\in {\mathfrak m}_{1},$ one gets, using (\ref{bracABC}), the following brackets:
$$
\begin{array}{lcl}
[u_{0},x]_{\mathfrak h} & = & 2(x^{1}_{1}H_{7} -x^{2}_{2}H_{8} + x^{2}_{1}H_{9} -x^{1}_{2}H_{10}),\\[0.4pc]
[e_{1},v]_{\mathfrak h} &  =  & v^{1}_{2}H_{1} - v^{3}_{1}H_{2} + v^{3}_{2}H_{3} + v^{1}_{2}H_{4} -v^{3}_{1}H_{5} + v^{3}_{2}H_{6} + v^{2}_{2}H_{7}\\[0.4pc]
 &  &  - v^{4}_{1}H_{8} + v^{4}_{2}H_{9} -v^{2}_{1}H_{10} + \sqrt{5}v^{1}_{2}H_{11}.
\end{array}
$$
It implies that $[u_{0},x]_{\mathfrak h}$ or $[e_{1},v]_{\mathfrak h}$ is zero if and only if $x$ is collinear to $u_{0}$ or $v$ is collinear to $e_{1}.$ Then the result follows from Proposition \ref{ptrans}.
\end{proof}

\begin{theorem}\label{conjB13} Let $\gamma_{u},$ $u = u(\theta),$ be a geodesic on $(B^{13},g)$ with slope angle $\theta = {\mathrm ang}(u,{\mathfrak m}_{0})\in [0,\frac{\pi}{2}].$ Then we have:
\begin{enumerate}
\item[{\rm (i)}] If $\gamma_{u}$ is a vertical geodesic, the points $\gamma_{u}(\frac{p\pi}{2}),$ $p\in {\mathbb Z},$ are isotropically conjugate to the origin and those $\gamma_{u}(p\pi)$ are not strictly isotropic.
\item[{\rm (ii)}] If $\theta\in ]0,\frac{\pi}{2}]$ and $u$ is orthogonal to $u_{0},$ $\gamma_{u}(\frac{t}{\sqrt{1 +\sin^{2}\theta}}),$ where
\begin{enumerate}
\item[{\rm (A)}] $t$ is a solution of the equation $\tan \frac{t}{2} = -\sin^{2}\theta\frac{t}{2}$
\noindent or
\item[{\rm (B)}] $t = 2p\pi,$ $p\in {\mathbb Z},$
\end{enumerate}
are conjugate points to the origin. In the first case, they are not strictly isotropic and in the second one, they are isotropic.
\item[{\rm (iii)}] If $\gamma_{u}$ is a horizontal geodesic, the points $\gamma_{u}(\frac{\sqrt{2}t}{2})$ as in {\rm (ii) A}, where $t$ is a solution of the equation $\tan \frac{t}{2} = -\frac{t}{2},$ are conjugate points to the origin but not isotropic.
\end{enumerate}
\end{theorem}
\begin{proof} From above lemma, the study of conjugate points along $\gamma_{u}$ can be reduced to consider $u$ as $u = u(\theta) = \cos \theta x + \sin \theta e_{1},$ where $x= x^{0}u_{0} + \sum_{i=1,2}(x^{i}_{1}u_{i} + x^{i}_{2}v_{i})\in {\mathfrak m}_{0}$ and $x_{0}^{2} + \sum_{i,j=1,2}(x^{i}_{j})^{2} = 1.$ Put $v = \cos\theta e_{1} - \sin\theta x.$ Then, from (\ref{bracABC}), one gets
\[
[u,v] = [x,e_{1}] = -x^{1}_{1}e_{2} -x^{2}_{1}e_{4} + x^{0}f_{1} + x^{1}_{2} f_{2} + x^{2}_{2}f_{4}
\]
and
$$
\begin{array}{lcl}
[[u,v],u] & = & x^{0}(\cos\theta(x^{0}e_{1} + x^{1}_{2}e_{2} + x^{2}_{2}e_{4} + x^{1}_{1}f_{2} +x^{2}_{1}f_{4}) - \sin\theta(u_{0}+H_{1} + H_{4} + \sqrt{5}H_{11}))\\[0.4pc]
 & & + x^{1}_{1}(\cos\theta(x^{1}_{1}e_{1} + x^{2}_{1}e_{3} + x^{1}_{2}f_{1} - x^{0}f_{2} -x^{2}_{2}f_{3})-\sin\theta(u_{1} + H_{10}))\\[0.4pc]
 & & + x^{2}_{1}(\cos\theta(x^{2}_{1}e_{1} - x^{1}_{1}e_{3} + x^{2}_{2}f_{1} + x^{1}_{2}f_{3} -x^{0}f_{4}) -\sin\theta(u_{2} + H_{8}))\\[0.4pc]
 & & + x^{1}_{2}(\cos\theta(x^{1}_{2}e_{1} - x^{0}e_{2} -x^{2}_{2}e_{3} -x^{1}_{1}f_{1} - x^{2}_{1}f_{3})-\sin\theta(v_{1} + H_{7}))\\[0.4pc]
 & & + x^{2}_{2}(\cos\theta(x^{2}_{2}e_{1} + x^{1}_{2}e_{3} -x^{0}e_{4} -x^{2}_{1}f_{1} + x^{1}_{1}f_{3}) - \sin\theta(v_{2} + H_{9})).
 \end{array}
 $$
Hence, it follows
\[
[[u,v],u]_{\mathfrak m} = v,\;\;\;\;[u,[u,v]]_{\mathfrak k} = \sin\theta(x^{0}(H_{1} + H_{4} + \sqrt{5}H_{11}) + x^{1}_{1}H_{10} + x^{2}_{1}H_{8} + x^{1}_{2}H_{7} + x^{2}_{2}H_{9}).
\]
Moreover, one gets
\begin{equation}\label{fB13}
[[u,[u,v]]_{\mathfrak k},u] =\sin^{2}\theta([u,v] + 6x^{0}f_{1}).
\end{equation}
If $\gamma_{u}$ is a vertical geodesic, $\rho =0$ and $u,v$ satisfy the hypothesis of Theorem \ref{cp} for $\lambda = 1.$ So, $\gamma_{u}(p\pi),$ $p\in {\mathbb Z},$ are not strictly isotropic conjugate points. Using again Lemma \ref{ltrans2}, we can take the geodesic $\gamma_{u},$ with $u = u_{0}.$ From (\ref{bracABC}) we get
\[
[u_{0},u_{1}] = 2H_{7},\;\;\;\; [[u_{0},u_{1}],u_{0}] = 4u_{1}.
\]
Then Theorem \ref{cp} (i) yields that $\gamma_{u}(\frac{p\pi}{2}),$ $p\in {\mathbb Z},$ are isotropically conjugate points to the origin. It proves (i).

For $\theta\in ]0,\pi/2],$ if $x^{0} = 0,$ or equivalently $u$ is orthogonal to $u_{0},$ (\ref{fB13}) implies that $[[u,[u,v]]_{\mathfrak k},u]$ is collinear to $[u,v].$ Then (ii) follows using Theorem \ref{cp} (ii) (B). Finally, using Lemma \ref{lsubmersion} and taking into account Remark \ref{rsymmetric} and that the base space of the canonical homogeneous Riemannian fibration is isometric to ${\mathbb C}P^{4}(2),$ the proof of (iii) is completed.
\end{proof}

\subsection{\sc The Wilking's example $W^{7}= (SO(3)\times SU(3))/U^{\bullet}(2)$}\indent

\vspace{0.2cm}

$U^{\bullet}(2)$ denotes the image of $U(2)$ under the
embedding $(\pi,\iota): U(2)\hookrightarrow SO(3)\times SU(3),$
where $\pi$ is the projection $\pi:U(2)\to U(2)/S^{1}\cong SO(3),$
being $S^{1}\subset U(2)$ the center of $U(2),$ and $\iota:\Lie{U}(2) \to \Lie{SU}(3)$ the
natural inclusion (see Section 4.2). Using the natural isomorphism between ${\mathfrak s}{\mathfrak o}(3)$ and ${\mathfrak
s}{\mathfrak u}(2),$ we can consider the Lie algebra ${\mathfrak s}{\mathfrak o}(3)\oplus
{\mathfrak s}{\mathfrak u}(3)$ of $SO(3)\times SU(3)$ as the subalgebra of ${\mathfrak
s}{\mathfrak u}(5)$ of matrices of the form
$$
X = \left (
\begin{array}{cc}
X_{1} & 0 \\
0 & X_{2}\\
\end{array} \right )\;\;\;\; X_{1}\in {\mathfrak
s}{\mathfrak u}(2),\;X_{2}\in {\mathfrak s}{\mathfrak u}(3).
$$
Then the Lie algebra ${\mathfrak u}^{\bullet}(2)$ of $U^{\bullet}(2)$ is given by
$$
{\mathfrak u}^{\bullet}(2) = \left\{(\pi_{*}A,A) = \left (
\begin{array}{cccc}
\pi_{*}A & \vline & 0 & 0\\ \hline
 0 & \vline & A & 0\\
0 & \vline & 0 & -\trace{A} \\
\end{array}
\right ) \mid \;\;\;\;\;A\in {\mathfrak u}(2) \right\},
$$
where $\pi_{*}$ denotes the differential map of $\pi.$ It implies that ${\mathfrak u}^{\bullet}(2) \cong \Delta({\mathfrak s}{\mathfrak u}(2))\oplus \RR Z_{0},$ where  $\Delta({\mathfrak s}{\mathfrak u}(2))$ is the Lie subalgebra in ${\mathfrak u}^{\bullet}(2)$ defined by $\Delta({\mathfrak s}{\mathfrak u}(2)) = \{(X,X)\mid X\in {\mathfrak s}{\mathfrak u}(2)\}$ and $Z_{0}\cong (0,Z_{0})$ is a generator of the centralizer of ${\mathfrak s}{\mathfrak u}(2)$ in ${\mathfrak s}{\mathfrak u}(3)\cong (0,{\mathfrak s}{\mathfrak u}(3)).$ On ${\mathfrak s}{\mathfrak u}(5)$ we take the bi-invariant inner product given by $\langle X,Y\rangle = -\frac{1}{2}{\rm trace}\;XY.$ $SO(3)$ as symmetric space is isometric to $\RR P^{3}$ and it can be expressed as the quotient $SO(3)\times SO(3)/\Delta(SO(3)),$ being $\Delta(SO(3))$ the diagonal of $SO(3)\times SO(3).$ Then we get the corresponding reductive orthogonal decomposition ${\mathfrak s}{\mathfrak u}(2) \oplus {\mathfrak s}{\mathfrak u}(2) = \Delta({\mathfrak s}{\mathfrak u}(2))\oplus {\mathfrak m}_{0},$ where ${\mathfrak m}_{0} = \{(X,-X)\mid X\in {\mathfrak s}{\mathfrak u}(2)\}.$

On the other hand, the quotient expression $SU(3)/S(U(2)\times U(1))$ for ${\mathbb C}P^{2}$ determines another reductive decomposition ${\mathfrak s}{\mathfrak u}(3) = ({\mathfrak s}{\mathfrak u}(2) \oplus \RR Z_{0})\oplus {\mathfrak m}_{1},$ ${\mathfrak m}_{1}\cong {\mathbb C}^{2}.$ Then, identifying ${\mathfrak m}_{1}$ with $(0,{\mathfrak m}_{1}),$ we have $[{\mathfrak u}^{\bullet}(2), {\mathfrak m}_{1}]\subset (0, [{\mathfrak s}{\mathfrak u}(2)\oplus \RR Z_{0},{\mathfrak m}_{1}]) \subset {\mathfrak m}_{1}.$ Hence, one gets
$$
\begin{array}{lcl}
{\mathfrak s}{\mathfrak o}(3)\oplus {\mathfrak s}{\mathfrak u}(3) & \cong & {\mathfrak s}{\mathfrak u}(2)\oplus {\mathfrak s}{\mathfrak u}(3) = ({\mathfrak s}{\mathfrak u}(2)\oplus {\mathfrak s}{\mathfrak u}(2))\oplus \RR Z_{0} \oplus {\mathfrak m}_{1}\\[0.4pc]
  & = & (\Delta({\mathfrak s}{\mathfrak u}(2))\oplus \RR Z_{0}) \oplus {\mathfrak m}_{0}\oplus {\mathfrak m}_{1}\cong {\mathfrak u}^{\bullet}(2)\oplus {\mathfrak m}_{0}\oplus {\mathfrak m}_{1} .
\end{array}
$$
Then ${\mathfrak s}{\mathfrak o}(3)\oplus {\mathfrak s}{\mathfrak u}(3) = {\mathfrak u}^{\bullet}(2)\oplus {\mathfrak m},$ where ${\mathfrak m} = {\mathfrak m}_{0} \oplus {\mathfrak m}_{1},$ is a reductive decomposition associated to $(SO(3)\times SU(3))/U^{\bullet}(2).$ Since $U^{\bullet}(2)$ is connected, it follows that ${\mathfrak m}_{0}$ and ${\mathfrak m}_{1}$ are $Ad(U^{\bullet}(2))$-invariant subspaces of ${\mathfrak m}.$

Any left-invariant metric on $SO(3)\times SU(3)$ is bi-invariant and, up to a scaling factor, they are given by the one-parameter family of inner products $\langle\cdot,\cdot\rangle_{s}:= s\langle\cdot,\cdot\rangle_{{\mathfrak s}{\mathfrak o}(3)\times {\mathfrak s}{\mathfrak o}(3)} + \langle\cdot,\cdot\rangle_{{\mathfrak s}{\mathfrak u}(3)\times {\mathfrak s}{\mathfrak u}(3)},$ for some $s>0,$ on ${\mathfrak s}{\mathfrak o}(3)\oplus {\mathfrak s}{\mathfrak u}(3).$ The induced metrics $g_{s}$ on $W^{7}$ which turn the projection $SO(3)\times SU(3)\to W^{7}$ into a Riemannian submersion are normal homogeneous and they are determined by the restriction of $\langle\cdot,\cdot\rangle_{s}$ to the orthogonal complement ${\mathfrak m}_{s}$ of ${\mathfrak u}^{\bullet}(2)$ on $({\mathfrak s}{\mathfrak o}(3)\oplus {\mathfrak s}{\mathfrak u}(3),\langle\cdot,\cdot\rangle).$ It can be expressed as ${\mathfrak m}_{s} = {{\mathfrak m}_{0}}_{s} \oplus {\mathfrak m}_{1},$ where ${{\mathfrak m}_{0}}_{s} = \{(X,-sX)\mid X\in {\mathfrak s}{\mathfrak u}(2)\}.$ Here, ${{\mathfrak m}_{0}}_{s}$ is a Lie triple system with associated totally geodesic diffeomorphic to $\RR P^{3}.$

An orthonormal basis $\{K_{1},K_{2},K_{3}, K_{4}; {u_{0}}_{s},{u_{1}}_{s},{v_{1}}_{s}; e_{1},e_{2},f_{1},f_{2}\}$ of $({\mathfrak s}{\mathfrak o}(3)\oplus
{\mathfrak s}{\mathfrak u}(3), \langle\cdot,\cdot\rangle_{s})$ adapted to the reductive decomposition ${\mathfrak
u}^{\bullet}(2)\oplus ({{\mathfrak m}_{0}}_{s}\oplus {\mathfrak m}_{1})$ is given as follows:
$$
\begin{array}{lcllcllcl}
K_{1} & = & \frac{1}{\sqrt{1+s}}(A_{1,2} + A_{3,4}),  & & K_{2}  & = & \frac{1}{\sqrt{1+s}}(B_{1,2} +
B_{3,4}), \\[0.4pc]
K_{3} & = & \frac{1}{\sqrt{1+s}}(C_{1,2} + C_{3,4}), & & K_{4} & = & \frac{1}{\sqrt{3}}(A_{3,4} + 2A_{4,5}),\\[0.4pc]
 {u_{0}}_{s} & = & \frac{1}{\sqrt{s(1+s)}}(A_{1,2} - s A_{3,4}), & & {u_{1}}_{s} & = & \frac{1}{\sqrt{s(1 +s)}}(B_{1,2} -s B_{3,4}),\\[0.4pc]
{v_{1}}_{s}& = & \frac{1}{\sqrt{s(1+s)}}(C_{1,2} - s C_{3,4}), & & e_{i} & = &  B_{i+2,5},\\[0.4pc]
f_{i} & = & C_{i+2,5},\;\; i = 1,2. & & & &
\end{array}
$$

\noindent (See \cite{GDN} for the brackets $[{\mathfrak m},{\mathfrak m}].)$ Each Riemannian space $(W^{7},g_{s})$ is isometric to the Aloff-Wallach space
$(M^{7}_{11},\tilde{g}_{t}),$ for $t=-\frac{3}{2s + 3}$
\cite{Wil}, and, in similar way than in above section, we have a homogeneous Riemannian fibration $\RR P^{3}\to W^{7} \to {\mathbb C}P^{2}.$ Moreover, since for orthonormal vectors $u,v\in {{\mathfrak m}_{0}}_{s}$ one gets
\[
\|[u,v]_{{\mathfrak m}_{s}}\|^{2} = \frac{4(1-s)^{2}}{s(1+s)},\;\;\; \|[u,v]_{{\mathfrak u}^{\bullet}(2)}\|^{2} = \frac{4}{1+s},
\]
formula (\ref{curnor}) implies that the fibers as totally geodesic submanifold are isometric to $\RR P^{3}(\frac{1 +s}{s}).$
\begin{lemma}\label{ltrans3} The linear isotropy action of $U^{\bullet}(2)$ on ${\mathcal S}\cap {{\mathfrak m}_{0}}_{s}$ and on ${\mathcal S}\cap {\mathfrak m}_{1}$ is transitive.
\end{lemma}
\begin{proof} Given $x= x^{0}{u_{0}}_{s} + x^{1}_{1}{u_{1}}_{s} + x^{1}_{2}{v_{1}}_{s}\in {{\mathfrak m}_{0}}_{s}$ and $v = \sum_{r=1}^{2}(v^{r}_{1}e_{r} + v^{r}_{2}f_{r})\in {\mathfrak m}_{1},$ one gets, using (\ref{bracABC}), the following brackets:
$$
\begin{array}{lcl}
[{u_{0}}_{s},x]_{{\mathfrak u}^{\bullet}(2)} & = & \frac{2}{\sqrt{1-s}}(-x^{1}_{2}K_{2} + x^{1}_{1}K_{3}),\\[0.4pc]
[e_{1},v]_{{\mathfrak u}^{\bullet}(2)} &  =  & \frac{1}{\sqrt{1+s}}(v^{1}_{2}K_{1} - v^{2}_{1}K_{2} + v^{2}_{2}K_{3} + \sqrt{3(1+s)}v^{1}_{2}K_{4}).
\end{array}
$$
It implies that $[{u_{0}}_{s},x]_{{\mathfrak u}^{\bullet}(2)}$ or $[e_{1},v]_{{\mathfrak u}^{\bullet}(2)}$ is zero if and only if $x$ is collinear to ${u_{0}}_{s}$ or $v$ is collinear to $e_{1}.$ Then the result follows from Proposition \ref{ptrans}.
\end{proof}

\begin{theorem}\label{conjW7} Let $\gamma_{u},$ $u = u(\theta),$ be a geodesic on $(W^{7},g_{s})$ with slope angle $\theta = {\mathrm ang}(u,{{\mathfrak m}_{0}}_{s})\in [0,\frac{\pi}{2}].$ Then we have:
\begin{enumerate}
\item[{\rm (i)}] If $\gamma_{u}$ is a vertical geodesic, the points $\gamma_{u}(2\sqrt{\frac{s}{1+s}}p\pi),$ $p\in {\mathbb Z},$ are isotropic conjugate to the origin and those $\gamma_{u}(2\sqrt{\frac{1+s}{s}}p\pi),$ are not strictly isotropic;
\item[{\rm (ii)}] If $\theta\in ]0,\frac{\pi}{2}]$ and $u$ is orthogonal to ${u_{0}}_{s},$ the points of the form $\gamma_{u}(\frac{t\sqrt{1+s}}{\sqrt{s +\sin^{2}\theta}}),$ where
\begin{enumerate}
\item[{\rm (A)}] $t$ is a solution of the equation $\tan \frac{t}{2} = -\sin^{2}\theta\frac{t}{2s},$
\noindent or
\item[{\rm (B)}] $t = 2p\pi,$ $p\in {\mathbb Z},$
\end{enumerate}
are conjugate to the origin. In the first case, they are not strictly isotropic and in the second one, they are isotropic.
\item[{\rm (iii)}] If $\gamma_{u}$ is a horizontal geodesic, the points $\gamma_{u}(t)$ in {\rm (ii) A}, where $t$ is a solution of the equation $\tan \frac{t}{2} = -\frac{t}{2s},$ are conjugate to the origin but not isotropic.
\end{enumerate}
\end{theorem}
\begin{proof} From the transitivity of the isotropy action on ${\mathcal S}\cap {\mathfrak m}_{1}$ proved in Lemma \ref{ltrans3}, the study of conjugate points along $\gamma_{u}$ can be reduced to consider $u$ as $u = u(\theta) = \cos \theta x + \sin \theta e_{1},$ where $x = x^{0}{u_{0}}_{s} + x^{1}_{1}{u_{1}}_{s} + x^{1}_{2}{v_{1}}_{s}\in {{\mathfrak m}_{0}}_{s},$ $(x^{0})^{2} + (x^{1}_{1})^{2} + (x^{1}_{2})^{2}= 1.$ Put $v = \cos\theta e_{1} - \sin\theta x.$ From (\ref{bracABC}), one gets
\[
[u,v] = [x,e_{1}] = \sqrt{\frac {s}{1+s}}(x^{1}_{1}e_{2} - x^{0}f_{1} - x^{1}_{2}f_{2})
\]
and
\[
[[u,v],u] = \frac{s}{1+s}(\cos\theta e_{1}-\sin\theta x) + \frac{\sqrt{s}}{1+s}\sin\theta (x^{0}(K_{1} + \sqrt{3(1+s)}K_{4}) + x^{1}_{1}K_{2} + x^{1}_{2}K_{3}).
\]
Hence, $[[u,v],u]_{{\mathfrak m}_{s}} = \frac{s}{1+s}v$ and $u$ and $v$ satisfy the conditions of Theorem \ref{cp} for $\lambda = \frac{s}{1+s}$ and $\rho = \frac{1}{1+s}\sin^{2}\theta(1 + 3(1+s)(x^{0})^{2}).$ Moreover, one gets
\begin{equation}\label{colW7}
[[u,[u,v]]_{{\mathfrak u}^{\bullet}(2)},u] = \frac{\sin^{2}\theta}{1+s}([u,v]- 3\sqrt{s(1+s)}x^{0}f_{1}).
\end{equation}
If $\gamma_{u}$ is vertical then $\rho = 0$ and so the points $\gamma_{u}(2\sqrt{\frac{1+s}{s}}p\pi),$ $p\in {\mathbb Z},$ are non-strictly isotropic conjugate points to the origin. Since in $\RR P^{3}(\frac{1+s}{s}),$ $\gamma_{u}(2\sqrt{\frac{s}{1+s}}p\pi)$ are (strictly) isotropic conjugate points to the origin, they also are isotropic in $(W^{7},g_{s}).$ It proves (i).

For $\theta\in ]0,\pi/2],$ if $u$ is orthogonal to ${u_{0}}_{s},$ one gets from (\ref{colW7}) that $[[u,[u,v]]_{{\mathfrak u}^{\bullet}(2)},u]$ is collinear to $[u,v]$ and $\rho = \frac{\sin^{2}\theta}{1 + s}.$ Then (ii) and (iii) follow using Theorem \ref{cp} (ii) (B) and Lemma \ref{lsubmersion} together with the fact that the base space of the canonical homogeneous Riemannian fibration is isometric to ${\mathbb C}P^{2}(1).$ \end{proof}

\noindent {\em Proof of Theorem \ref{tmean2}.} It follows using Propositions \ref{mean1} and \ref{pmean2} and Lemma \ref{lB13}.

\noindent {\em Proof of the Chavel's conjecture.} From Theorems \ref{conj}, \ref{conjB13} and \ref{conjW7}, one obtains, for the cases (ii)-(iv), (vi) and (vii) in Theorem \ref{tmean2}, the existence of non-isotropically conjugate points to the origin along any horizontal geodesic. Moreover, for the  Berger space $B^{7},$ I. Chavel \cite{Ch} finds non-isotropic conjugate points to the origin.

\end{document}